\documentclass[reqno,oneside,12pt]{amsart}

\usepackage[danish,english,french,german]{babel} 
\usepackage[utf8]{inputenc} 
\usepackage[T1]{fontenc} 
\usepackage{lmodern}
\usepackage{amsmath,amsthm,amssymb,amsfonts,mathrsfs,latexsym} 
\usepackage{enumerate}
\usepackage{graphicx} 
\usepackage{verbatim} 
\usepackage[all]{xy} 
\usepackage[pagebackref, colorlinks, linkcolor=red, citecolor=blue, urlcolor=blue, hypertexnames=true]{hyperref}  
\usepackage{multirow}
\usepackage{color}
\usepackage{tikz}
\usetikzlibrary{matrix,arrows}

\makeatletter
\newtheorem*{rep@theorem}{\rep@title}
\newcommand{\newreptheorem}[2]{%
\newenvironment{rep#1}[1]{%
 \def\rep@title{#2 \ref{##1}}%
 \begin{rep@theorem}}%
 {\end{rep@theorem}}}
\makeatother

\newtheorem{thm}{Theorem}[section]
\newtheorem{thmx}{Theorem}
\newtheorem{corx}[thmx]{Corollary}

\newreptheorem{thm}{Theorem}
\newtheorem{lem}[thm]{Lemma}
\newtheorem{prop}[thm]{Proposition}
\newtheorem{cor}[thm]{Corollary}
\newtheorem*{thm*}{Theorem}
\newtheorem*{problem*}{Problem}
\newtheorem{question}[thm]{Question}

\newtheorem*{claim*}{Claim}
\theoremstyle{definition}
\newtheorem{defi}[thm]{Definition}
\newtheorem{exam}[thm]{Example}

\newtheorem{rem}[thm]{Remark}

\newenvironment{claimproof}[1]{\textit{Proof of Claim. }}{\hfill $\blacksquare$}

\newcommand{\abs}[1]{\lvert#1\rvert}

\newcommand{\N}{\mathbb{N}}

\newcommand{\Z}{\mathbb{Z}}
\newcommand{\F}{\mathbb{F}}

\renewcommand{\epsilon}{\varepsilon}
\renewcommand{\phi}{\varphi}
\renewcommand{\tilde}{\widetilde}
\renewcommand{\hat}{\widehat}

\renewcommand{\bar}{\overline}

\newcommand{\ol}{\overline}

\DeclareMathOperator{\supp}{supp}

\DeclareFontFamily{U}{mathx}{\hyphenchar\font45}
\DeclareFontShape{U}{mathx}{m}{n}{
      <5> <6> <7> <8> <9> <10>
      <10.95> <12> <14.4> <17.28> <20.74> <24.88>
      mathx10
      }{}
\DeclareSymbolFont{mathx}{U}{mathx}{m}{n}
\DeclareFontSubstitution{U}{mathx}{m}{n}
\DeclareMathAccent{\widecheck}{0}{mathx}{"71}
\DeclareMathAccent{\wideparen}{0}{mathx}{"75}

\numberwithin{equation}{section}

\begin{document}
\selectlanguage{english} 
\date{\today}
\thanks{The first and third authors were partially supported by the DGI-MINECO and European Regional Development Fund, jointly, through the grant MTM2017-83487-P. Also, they acknowledge support from the Generalitat de Catalunya through the grant 2017-SGR-1725. The second author was partially supported by Deutsche Forschungsgemeinschaft (SFB 878) and the Alexander von Humboldt Foundation. The third author is also supported by Beatriu de Pinos programme (BP-2017-0079). The fourth author has received funding from the European Research Council (ERC) under the European Union's Horizon 2020 research and innovation programme (grant agreement no. 677120-INDEX)}


\title{The type semigroup, comparison and almost finiteness for ample groupoids}
\author{Pere Ara}
\address{Pere Ara: Departament de Matemàtiques, Universitat Autònoma de Barcelona, 01893 Cerdanyola del Vallès (Barcelona), Spain, and}
\address{Centre de Recerca Matemàtica, Edifici C, Campus de Bellaterra, 01893 Cerdanyola del Vallès (Barcelona), Spain}
\email{para@mat.uab.cat}

\author{Christian Bönicke}
\address{Christian Bönicke: School of Mathematics and Statistics, University of Glasgow, University Gardens, Glasgow, G12 8QQ, UK}
\email{christian.bonicke@glasgow.ac.uk}

\author{Joan Bosa}
\address{Joan Bosa: Departament de Matemàtiques, Universitat Autònoma de Barcelona, 01893 Bellaterra (Barcelona), Spain and BGSMath}
\email{jbosa@mat.uab.cat}

\author{Kang Li}
\address{Kang Li: Institute of Mathematics of the Polish Academy of Sciences, \'{S}niadeckich 8, 00-656 Warsaw, Poland}
\email{kli@impan.pl}

\begin{abstract} We prove that a minimal second countable ample groupoid has dynamical comparison if and only if its type semigroup is almost unperforated. Moreover, we investigate to what extent a not necessarily minimal almost finite groupoid
	has an almost unperforated type semigroup. Finally, we build a bridge between coarse geometry and topological dynamics by characterizing almost finiteness of the coarse groupoid in terms of a new coarsely invariant property for metric spaces, which might be of independent interest in coarse geometry. As a consequence, we are able to construct new examples of almost finite principal groupoids lacking other desirable properties, such as amenability or even a-T-menability. This behaviour is in stark contrast to the case of principal transformation groupoids associated to group actions. 
\end{abstract}
\date{\today}
\maketitle
\section*{Introduction}
Semigroups of equidecomposability types have been of interest ever since Tarski's seminal work on the dichotomy between amenability and paradoxicality for discrete groups. Recently, such a type semigroup has been introduced as a new invariant for ample groupoids by the second and fourth author in \cite{boenicke_li_2018} and independently in \cite{rainone_sims}. In this very general framework, the type semigroup has attracted significant interest for both the role it plays in the study of finitely generated conical refinement monoids \cite{ABPS19}, as well as its connection to the structure theory of the associated reduced groupoid $\mathrm{C}^*$-algebra.
In particular, the following dichotomy result was proved in \cite{boenicke_li_2018, rainone_sims}: If the type semigroup $S(G)$ of a minimal topologically principal ample groupoid $G$ with compact unit space is almost unperforated, then its reduced groupoid $C^*$-algebra $C_r^*(G)$ is a simple $C^*$-algebra, which is either stably finite or strongly purely infinite.




Consequently, it is a natural question to ask for conditions under which the type semigroup is almost unperforated. This is indeed the situation for all the monoids described in \cite{ABPS19}.
However, one can also build groupoids whose type semigroup is not almost unperforated via the (usually non-amenable) groupoids associated to the separated graphs defined in \cite{AE14}.

The first main result of this article is a dynamical analogue of a celebrated result by R\o rdam in \cite{MR2106263} on the equivalence between strict comparison and almost unperforation of the Cuntz semigroup for unital simple separable exact $\mathrm{C}^*$-algebras:
\begin{thmx}\label{Thm:A}
	Let $G$ be a second countable minimal ample groupoid. Then $G$ has dynamical comparison if and only if its type semigroup $S(G)$ is almost unperforated.
\end{thmx}
In fact, we can even show that the type semigroup $S(G)$ is almost unperforated for a \emph{$\sigma$-compact} minimal ample groupoid $G$ which has dynamical comparison. As a very special case we have:
\begin{corx}
Let $\Gamma$ be a countable discrete group acting minimally on a locally compact second countable totally disconnected space $X$. Then the transformation group $X\rtimes \Gamma$ has dynamical comparison if and only if its type semigroup $S(X\rtimes \Gamma)$ is almost unperforated.
\end{corx}
The novelty here lies in a rather elementary approach, which allows us to drop any freeness or amenability assumptions that were crucial in previous attempts to prove such a result for transformation groups (see \cite{1710.00393} and \cite[Corollary~6.3]{Ma18}).
A great range of examples has been constructed in \cite{1712.05129}, where the authors prove that every action of a countable group with local subexponential growth on a zero dimensional compact metric space has dynamical comparison.

We then study dynamical comparison and almost unperforation of the type semigroup in the context of other important structural properties of the groupoid. In contrast to the above result, we do not limit ourselves to the minimal case and investigate two different situations: In the infinite case, i.e. when there are no non-trivial invariant measures on the unit space, we show dynamical comparison is equivalent to pure infiniteness of the groupoid, extending earlier results of Ma \cite{Ma18} (see Section~2.1 for details).

On the other end of the spectrum, we consider almost finite (not necessarily minimal) groupoids as introduced by Matui in \cite{MR2876963}. Recall that in \cite{1710.00393}, David Kerr specialises to almost finite group actions and proposes that almost finiteness might play a role in topological dynamics analogous to the role $\mathcal{Z}$-stability does for simple $\mathrm{C}^*$-algebras (see \cite[Theorem~12.4]{1710.00393} and see also \cite{2002.12221,MW}).
In particular, in the setting of free minimal group actions on zero-dimensional spaces he shows that almost finiteness always implies almost unperforation of the type semigroup, which also implies dynamical comparison (see \cite[Theorem~13.3]{1710.00393}). In the subsequent work \cite{1807.04326}, Kerr and Szabó prove that a minimal free action of amenable groups on compact metrizable zero-dimensional spaces is almost finite if and only if the action has dynamical comparison.

Studying almost unperforation for the type semigroup of non-minimal almost finite group\-oids leads to new complications. 
The main obstacle is the different behaviour of almost finiteness and almost unperforation when passing to open invariant subsets of the unit space.
To circumvent this problem, we call a groupoid $G$ {\it strongly almost finite} if every restriction of $G$ by a compact open subset of the unit space is almost finite in the sense of Matui. It is worth noticing that when the groupoid is minimal and has a compact unit space, strong almost finiteness agrees with the usual almost finiteness (see Proposition \ref{Prop:MoritaInvariance}). The second main result of these notes is the following:

\begin{thmx}\label{Thm:B}
If $G$ is a strongly almost finite ample groupoid, then it has stable dynamical comparison. In particular, its type semigroup $S(G)$ is almost unperforated.
\end{thmx}
Note that Theorem \ref{Thm:B} does not require $G$ to be second countable.

In the final section, we establish a link between regularity properties in topological dynamics (which are in turn inspired by their counterparts in the structure theory of nuclear $\mathrm{C}^*$-algebras) and a new tiling property in coarse geometry. In fact, inspired by recent results on the structure of amenable groups in \cite{DHZ}, we introduce a tiling property, which is a strong version of amenability for metric spaces, which asserts that the space can be tiled by uniformly bounded F\o lner sets of arbitrary invariance (see Definition~\ref{tilings of arbitrary invariance}). In order to explain the promised connection to coarse geometry, we recall that for every discrete metric space $X$ with \emph{bounded geometry} (i.e., for any $R>0$ there is a uniform upper bound on the cardinalities of all the $R$-balls in $X$), one can construct a principal locally compact
$\sigma$-compact ample groupoid $G(X)$ with unit space $\beta X$ (see \cite[Proposition~3.2]{zbMATH01761014}). The groupoid $G(X)$ is called the \emph{coarse groupoid} of $X$ and it is \emph{not} minimal in general. It is well-known that this groupoid reflects many interesting properties in coarse geometry. For instance, $X$ has Yu's property A if and only if $G(X)$ is an amenable groupoid (see \cite[Theorem 5.3]{zbMATH01761014}).

We show that our new tiling property is invariant under coarse equivalences and provide a link to the main results of this article by proving:
\begin{thmx}[see Theorem~\ref{Thm:AlmostFiniteCoarseGroupoid}]\label{Thm:C}
	
Let $X$ be a bounded geometry metric space and $G(X)$ be its coarse groupoid. Then the following are equivalent:
	\begin{enumerate}
		\item $G(X)$ is almost finite,
		\item $X$ admits tilings of arbitrary invariance.
	\end{enumerate}
	In particular, $G(X)$ is strongly almost finite if and only if every subspace of $X$ admits tilings of arbitrary invariance.
\end{thmx}

In particular, we obtain from Theorem \ref{Thm:B} and Theorem~\ref{Thm:C}  that the type semigroup $S(G(X))$ is almost unperforated for any bounded geometry metric space $X$ such that every subspace of $X$ admits tilings of arbitrary invariance
(see Corollary \ref{coarse groupoid+almost unp}). Moreover, Theorem~\ref{Thm:C} can be used to provide a range of new examples of (strongly) almost finite groupoids (see Section 4 for details). In particular, the result allows us to construct groupoids that exhibit a behaviour which cannot be witnessed in the setting of transformation groupoids. Specifically, we elaborate on the subtle relationship between almost finiteness and amenability and provide new examples of principal almost finite groupoids which are non-amenable, in fact not even a-T-menable. This answers a query of Yuhei Suzuki (see \cite[Remark 3.7]{Suzuki})\footnote{Non-amenable minimal almost finite groupoids are independently constructed by Gabor Elek in \cite{1812.07511}.}.
While for the purposes of this article we only use our new tiling property for metric spaces to obtain interesting examples of groupoids, we believe that it might be of independent interest in coarse geometry as well.




 We briefly outline the contents of this paper. In section \ref{sec:Prelim}, we recall the necessary definitions concerning groupoids, their type semigroups, and their connection to groupoid homology. In the second section, we study dynamical comparison, and its relation to almost unperforation of the type semigroups. The main result obtained in this section is Theorem \ref{Thm:A}. In section \ref{sec:AF}, we focus our study on almost finite groupoids. In order to ease reading it, we have divided this part into two subsections. In the first subsection, we recall the definition of almost finiteness and establish that it is invariant under stable isomorphism. In the second subsection, we prove Theorem \ref{Thm:B}, and describe some implications on the relation between the type semigroup and the positive cone of the groupoid homology. 
As mentioned before, the groupoids in Theorem \ref{Thm:B} may not be second countable. In order to achieve this level of generality we require some technical tools concerning extensions of Borel measures, which are developed in Appendix \ref{sect:ext-measures}. We finish the main body of the paper with section \ref{Section:CoarseGeometry}, in which we introduce our new tiling property for metric spaces and prove Theorem \ref{Thm:C}. As a consequence, we provide new examples of (strongly) almost finite groupoids and particularly construct non-amenable almost finite groupoids in Corollary \ref{cor:almost finite non-amen}. Finally, we use some of the methods developed in this article to give a short and conceptual proof of a classical result by Block and Weinberger, characterizing (non-)amenability of metric spaces in terms of uniformly finite homology (see Corollary \ref{Cor:BlockWeinberger}).

\section{Preliminaries on groupoids and the type semigroup}\label{sec:Prelim}
Let us start reviewing the terminology and notation related to groupoids that we will use throughout the text.
Given a groupoid $G$ we will denote its unit space by $G^{(0)}$ and write $r,s:G\rightarrow G^{(0)}$ for the range and source maps, respectively. \textbf{Throughout the paper, all groupoids are always assumed to be equipped 
with a locally compact, Hausdorff topology making all the structure maps continuous.} A groupoid $G$ is called \textit{étale} if the range map, regarded as a map $r:G\rightarrow G$, is a local homeomorphism. It is called \textit{ample} if it is étale and the unit space $G^{(0)}$ is totally disconnected.
In that case $G$ admits a basis for its topology consisting of compact and open \textit{bisections}, i.e. compact and open subsets $V\subseteq G$ such that the restrictions of the source and range maps to $V$ are homeomorphisms onto their respective images. Note that a compact open bisection $V\subseteq G$ gives rise to a homeomorphism $\theta_V:s(V)\rightarrow r(V)$.

For two subsets $A,B\subseteq G$ we will consider their product
$$AB=\lbrace ab\in G\mid a\in A, b\in B, s(a)=r(b)\rbrace.$$
If $B=\lbrace x\rbrace$ for a single element $x\in G^{(0)}$ we will omit the braces and just write $Ax$.

For a subset $D\subseteq G^{(0)}$ we use the standard notations $G_{D}=\lbrace g\in G\mid s(g)\in D\rbrace$, and $G^D=\{ g\in G\mid r(g)\in D\}$. When $D=\{x\}$ consists of a single point, we will adopt the customary abuse of notation and just write $G_x$ instead of $G_{\{x\}}$ and similarly $G^x$ instead of $G^{\{x\}}$.

\subsubsection*{Constructions of groupoids}
Let us briefly review some constructions of groupoids that we will use frequently throughout the article.

Given a subset $D\subseteq G^{(0)}$ we can consider the subgroupoid $G|_D:=G_D\cap G^D$ called the \textit{restriction} of $G$ to $D$. If $D=\{x\}$ consists of a single point, then $G|_D$ is a group, called the isotropy group at $x\in G^{(0)}$. 
Assembling all the isotropy groups of $G$ gives rise to the \textit{isotropy} groupoid $\mathrm{Iso}(G)=\lbrace g\in G\mid s(g)=r(g)\rbrace$ of $G$.

Another construction appearing in this article is the product of two groupoids: Given groupoids $G$ and $H$, their product $G\times H$ can be equipped with a groupoid structure by multiplying and inverting elements componentwise. If $G$ and $H$ are étale (resp. ample), then their product is also étale (resp. ample).

\subsubsection*{Dynamical notions}
We say that a set $D\subseteq G^{(0)}$ is $G$\textit{-invariant} if for every $g\in G$ we have $r(g)\in D \Leftrightarrow s(g)\in D$. Note that in this case $G|_D=G^D=G_D$ so when talking about the restriction of $G$ to an invariant subset we will often just write $G_D$. Thus, the notation $G|_D$ introduced above will typically indicate the restriction to a non-invariant subset. Note that if $D$ is $G$-invariant, then so is its complement $G^{(0)}\setminus D$.
If there are no proper non-trivial closed $G$-invariant subsets of $G^{(0)}$ we say that $G$ is \textit{minimal}.

We say that $G$ is \textit{principal} if Iso$(G)=G^{(0)}$.
Finally, $G$ is called \textit{effective} if the interior of Iso$(G)$ coincides with $G^{(0)}$. This is connected with the notion of {\it topologically principal}, which means that the set of points of $G^{(0)}$ 
with trivial isotropy group is dense in $G^{(0)}$. If $G$ is second countable and effective, then $G$ is topologically principal. If $G$ is Hausdorff and topologically principal, then $G$
is effective (\cite[Proposition 3.6]{Re80}).  


\subsection{Type semigroup}

The type semigroup of an ample groupoid was introduced and studied in \cite{boenicke_li_2018,rainone_sims}. In this section, we recall its definition and study some of its basic properties. 

\begin{defi}
	Given an ample groupoid $G$, we define an equivalence relation $\sim_G$ on $C_c(G^{(0)},\Z)^+$ by declaring $f_1\sim_G f_2$ if there exist compact open bisections $W_1,\ldots,W_n$ of $G$ such that $f_1=\sum_{i=1}^n 1_{s(W_i)}$ and $f_2=\sum_{i=1}^n 1_{r(W_i)}$. We define the type semigroup associated to $G$ by $$S(G):=C_c(G^{(0)},\Z)^+/\sim_G.$$
 We will write $[f]$ for the equivalence class of a function $f\in C_c(G^{(0)},\Z)^+$, and equip $S(G)$ with the addition induced by pointwise addition in $C_c(G^{(0)},\Z)^+$. In particular, $S(G)$ contains the class of the zero function as a neutral element and can be equipped with the algebraic preorder (i.e. $x\leq y$ in $S(G)$ if and only if there exists an element $z\in S(G)$ such that $x+z=y$). 
\end{defi}

The type semigroup is clearly an isomorphism invariant for groupoids and it was shown in \cite{rainone_sims} that it is also invariant under all the various (equivalent notions) of groupoid equivalence. This observation will be important later.

Recall, that a commutative monoid $S$ is called \textit{conical}, if for all $x,y\in S$, $x+y=0$ only when $x=y=0$.
We say that $S$ is a \textit{refinement monoid} if for all $a,b,c,d\in S$ such that $a+b=c+d$ there exist $w,x,y,z\in S$ such that $a=w+x$, $b=y+z$, $c=w+y$, and $d=x+z$.
It is straightforward to verify, that $S(G)$ is always a conical refinement monoid.

An important part of the structure of a preordered monoid $S$ is the collection of its order units. Recall, that a non-zero element $u\in S$ is called an \textit{order unit}, provided that for every $x\in S$ there exists $n\in\N$ such that $x\leq nu$. We will write $S^*$ for the collection of all order units in $S$. The monoid $S$ is called \textit{simple}, provided that every non-zero element of $S$ is an order unit, in other words $S=S^*\cup\lbrace 0\rbrace$. It has already 
been observed in \cite[Lemma~5.9]{boenicke_li_2018} that the type semigroup $S(G)$ of an ample groupoid $G$ is simple, provided that $G$ is minimal. The following Lemma extends this observation by identifying all the order units. To understand the statement and its proof, recall that a set $D\subseteq G^{(0)}$ is called $G\mathit{-full}$ if $r(GD)=G^{(0)}$ (i.e. for every element $x\in G^{(0)}$ there exists a $g\in G$ with $s(g)\in D$ and $r(g)=x$).

\begin{lem}\label{Lemma:Order units}
	Let $G$ be an ample groupoid and $[f]\in S(G)$. Then $[f]$ is an order unit if and only if $supp(f)$ is $G$-full.
	In particular, $G$ is minimal if and only if $S(G)$ is simple.
\end{lem}
\begin{proof}
	Suppose $[f]\in S(G)$ is an order unit. Let $x\in G^{(0)}$ and $K$ be a compact open set containing $x$. Then there exists some $n\in\N$ such that $[1_{K}]\leq n[f]$. Let $m\in \Z^+$ be the maximal value attained by $f$. Then clearly $[1_{K}]\leq [nm1_{supp(f)}]$ and hence there exist compact open bisections $V_1,\ldots, V_k$ such that
	$K=\bigsqcup_{i=1}^k r(V_i)$ and $supp(f)\supseteq \bigcup_{i=1}^k s(V_i).$ It follows that $r(Gsupp(f))\supseteq K$ and we are done.
	For the converse we may proceed as in the proof of \cite[Lemma~5.9]{boenicke_li_2018}.
\end{proof}

Let us also identify the order ideals of the type semigroup.
Recall, that an \textit{order ideal} of a monoid $S$ is a submonoid $I$ such that for all $x,y\in S$, we have $x+y\in I$ if and only if $x,y\in I$.
\begin{lem}\label{Lemma:OrderIdeals}
	Let $G$ be an ample groupoid. If $I$ is an order ideal in $S(G)$, then there exists an open invariant subset $U\subseteq G^{(0)}$, such that $I\cong S(G_{U})$. 
\end{lem}
\begin{proof}
	Suppose $I$ is an order ideal in $S(G)$. Then $U:=\bigcup\lbrace supp(f)\mid [f]\in I\rbrace$ is an open subset of $G^{(0)}$. To see that it is invariant, let $g\in G$ such that $s(g)\in U$. Then there exists $[f]\in I$ such that $s(g)\in supp(f)$. Now $f$ can be written as $f=\sum_i 1_{A_i}$ for suitable clopen sets $A_i$ and $s(g)$ must be contained in one of these. Since $I$ is an order ideal, each $[1_{A_i}]\in I$. Now let $V$ be a compact open bisection containing $g$ such that $s(V)\subseteq A_i$. Upon refining the representation of $f$ if necessary, we may assume $s(V)=A_i$. Since $[1_{r(V)}]=[1_{A_i}]\in I$ we get $r(g)\in U$.
	
	Now let $J$ denote the ideal of $S(G)$ generated by all the elements of $S(G)$ which can be represented by a function whose support is contained in $U$. Then we clearly have $I\subseteq J$. For the converse inclusion take any $[f]\in S(G)$ such that $supp(f)\subseteq U$. Since the support of $f$ is compact, we may find finitely many functions $f_1,\ldots, f_n$ such that $[f_i]\in I$ with $supp(f)\subseteq \bigcup supp(f_i)$. In particular, we have $f\leq \sum_i n_i f_i$ for suitably large $n_i\in \N$ and hence $[f]\leq \sum_i n_i [f_i]\in I$. Since $I$ is an order ideal, this implies $[f]\in I$ as desired. Clearly, we have $J\cong S(G_U)$.
\end{proof}

Once we have an order ideal $I$ in a monoid $S$ one can define a congruence on $S$ by declaring $x\sim y$ if there exist elements $a,b\in I$ such that $x+a=y+b$. Then $S/I:=S/\sim$ can be equipped canonically with a monoid structure induced by $S$.
To identify the quotients of the type semigroup, note that the set $G^a$ of all compact open bisections of an ample Hausdorff groupoid $G$ forms a Boolean inverse semigroup and the type semigroup $S(G)$ can be canonically identified with the type monoid Typ$(G^a)$ of this inverse semigroup (see \cite[Proposition~7.3]{ABPS19}).

We shall also need the following construction: For two compact open bisections $V_1,V_2$ in $G$ let $E= s(V_1)\setminus (s(V_1)\cap s(V_2))$ and $F=r(V_1)\setminus (r(V_1)\cap r(V_2))$ and define
$$V_1\bigtriangledown V_2= FV_1E\cup V_2.$$
Then $V_1\bigtriangledown V_2$ is a compact open bisection in $G$.

The proof of the following result is essentially contained in \cite[Lemma~5.5]{LV19}. We spell out a sketch of the proof for the readers convenience.
\begin{prop}
	Let $I\subseteq S(G)$ be an order ideal. If $U$ is the open $G$-invariant subset of $G^{(0)}$ corresponding to $I$, and $D=G^{(0)}\setminus U$ its complement, then the canonical map $S(G)\rightarrow S(G_D)$ induced by restriction of functions gives rise to an isomorphism $S(G)/I\cong S(G_D)$.
\end{prop}
\begin{proof}
	Upon identifying $S(G)$ with Typ$(G^a)$, the result follows from \cite[Theorem 4.3.2]{W17} once we realize that the canonical semigroup homomorphism $G^a\rightarrow (G_D)^a$ is surjective. To see this proceed as follows: If $V\subseteq G_D$ is a compact open bisection, then by definition of the induced topology and using the fact that being compact does not depend on the ambient space, we can find finitely many compact open bisections $U_1,\ldots, U_n$ in $G$ such that $V=\bigcup_i U_i\cap G_D$. Then $\bigtriangledown_i U_i$ is a compact open bisection in $G$ such that $(\bigtriangledown_i U_i) \cap G_D=V$.
\end{proof}

\subsection{Groupoid homology and its relation with $S(G)$}

Let us now turn our attention to understand the relationship between the type semigroup of an ample groupoid $G$ and the positive cone $H_0(G)^+$ of the (integral) groupoid homology $H_0(G)$. We refer the reader to \cite[Section~3]{MR2876963} for the relevant definitions. The relevant property here is cancellation: Recall that we say that a semigroup $S$ is \textit{cancellative} if for $a,b,c,\in S$ satisfying $a+c=b+c$, it follows that $a=b$.

\begin{lem}
	Let $G$ be an ample groupoid with a compact unit space. Then the quotient map $C(G^{(0)},\Z)\rightarrow H_0(G)$ induces a surjective semigroup homomorphism
	$$ S(G)\rightarrow H_0(G)^+$$
\end{lem}
\begin{proof}
	We need to show that the map is well-defined. Suppose $f,g\in C(G^{(0)},\Z)^+$ such that $f\sim g$ in $S(G)$. We will show that $f-g\in im(\partial_1)$, where $\partial_1:C_c(G,\Z)\rightarrow C(G^{(0)},\Z)$ is the differential map from the chain complex defining groupoid homology. This immediately implies $[f]=[g]$ in $H_0(G)^+$.
	Since $f\sim g$ in $S(G)$ we can find bisections $V_1,\ldots, V_n$ such that 
	$$f=\sum\limits_{i=1}^n 1_{s(V_i)}\text{ and } g=\sum\limits_{i=1}^n 1_{r(V_i)}.$$
	Let $h:=\sum_{i=1}^n 1_{V_i}\in C_c(G,\Z)$. Then 
	$$\partial_1(h)=\sum\limits_{i=1}^n s_*(1_{V_i})-r_*(1_{V_i})=\sum\limits_{i=1}^n 1_{s(V_i)}-1_{r(V_i)}=f-g$$
	as desired.
\end{proof}

Before the next result, let us recall the construction of the universal cancellative abelian semigroup.
Let $S$ be an abelian semigroup with $0\in S$, and consider the equivalence relation on $S$ given by $x\sim y$ if there exists an element $z\in S$ such that $x+z=y+z$.
Then $\sim$ is an equivalence relation and $C(S):=S/\sim$ is a cancellative abelian semigroup with the (universal) property, that for every homomorphism $\Phi:S\rightarrow P$ into a cancellative abelian semigroup $P$ there exists a unique homomorphism $C(\Phi):C(S)\rightarrow P$ such that $C(\Phi)([s])=\Phi(s)$.

\begin{prop}\label{Prop:Cancellative hull of type semigroup}
	Let $G$ be an ample groupoid with compact unit space. Then the canonical map $S(G)\rightarrow H_0(G)^+$ induces an isomorphism of cancellative abelian semigroups.
	$$C(S(G))\rightarrow H_0(G)^+$$
\end{prop}
\begin{proof}
	By universality, one can build a well-defined surjective homomorphism  $C(S(G))\rightarrow H_0(G)^+$. Hence, it remains to check its injectivity. Let $f,g\in C(G^{(0)},\Z)^+$ such that $[f]=[g]$ in $H_0(G)^+$.
	Then $f-g\in im(\partial_1)$, i.e. there exists a function $h\in C_c(G,\Z)$ such that
	$$f-g=\partial_1(h)=s_*(h)-r_*(h)$$
	This implies that $f+r_*(h)=g+s_*(h)$.
	Since $h$ is compactly supported and $G$ is ample, we can write $h=\sum\limits_{i=1}^m 1_{V_i}$ for appropriately chosen compact open bisections $V_1,\ldots, V_m$, which implies that $r_*(h)\sim s_*(h)$. This concludes the proof since if $x:=[s_*(h)]=[r_*(h)]$, then $[f]+x=[g]+x$ in $S(G)$ and hence $[f]=[g]$ in $C(S(G))$.
\end{proof}

\section{Dynamical comparison}
In this section we study the relation between almost unperforation of the type semigroup and dynamical comparison, an important regularity property.
Since our definition is rather general and in particular not limited to minimal groupoids we need to recall some facts about (possibly infinite) Borel measures for locally compact Hausdorff spaces.

For a topological space $X$, we denote by $UM(X)$  the cone  of positive Borel measures on $X$. For a given Borel subset $B$ of $X$, 
the convex subset $UM(X,B)\subseteq UM(X)$ consists of those $\mu \in UM(X)$ such that $\mu (B) =1$. If $X$ is further locally compact and Hausdorff, we denote by $UM_c(X)$ the cone of all the positive 
regular Borel measures $\mu$ on $X$ such that $\mu (K)<\infty$ for all compact sets $K$ of $X$. By \cite{Rud87}, if $X$ is in addition $\sigma$-finite, then $UM_c(X)$ can be identified with the positive part of the dual space of 
the space $C_c(X)$. Finally, if $X$ is compact, we will denote by $M(X)$ the compact convex set of all the positive regular Borel probability measures on $X$, which is isomorphic to the positive part of the 
unit ball of the dual of $C(X)$. 

Now let $G$ be an \'etale groupoid (so that $G^{(0)}$ is a locally compact Hausdorff space), and recall that a Borel measure $\mu$ on $G^{(0)}$ is called \textit{$G$-invariant} if $\mu(s(V))=\mu(r(V))$ for every open 
bisection $V\subseteq G$. Slightly abusing notation, we write $UM(G)$ for the subcone of $UM(G^{(0)})$ of all the invariant positive Borel measures on $G^{(0)}$.  Similarly, we will write $UM_c(G)$ for the subcone of $UM_c(G^{(0)})$ 
consisting of all the invariant positive regular Borel measures $\mu$ on $X$ such that $\mu (K)<\infty$ for all compact subsets $K$ of $G^{(0)}$. If in addition $G^{(0)}$ is compact, we denote by $M(G)$ the compact 
convex set of invariant positive regular Borel probability measures on $G^{(0)}$.  

We now introduce a version of dynamical comparison which also works in the non-minimal case.

\begin{defi}
	\label{def:dyn-comparison-nonminimal}
	Let $G$ be an ample groupoid.
	For two compact open subsets $A,B\subseteq G^{(0)}$ we say that $A$ \textit{is subequivalent to} $B$ and write $A\precsim B$, if there exist finitely many compact open bisections $V_1,\ldots,V_n$ of $G$ such that $A=\bigsqcup_{i=1}^n s(V_i)$ and the sets $\{r(V_i)\}_{i=1}^n$ 
	are pairwise disjoint subsets of $B$.
		
	We say that $G$ has \textit{dynamical comparison} if for all nonempty compact open subsets $A,B\subseteq G^{(0)}$ such that $A\subset r(GB)$
	and satisfying $\mu(A)<\mu(B)$ for every $\mu \in UM(G)$ such that $0< \mu (B) <\infty$, we have $A\precsim B$.
	We say that $G$ has {\it stable dynamical comparison} if $G^m$ has dynamical comparison for all $m\geq 1$, where $G^m$ denotes the product groupoid $G\times \mathcal R_m$.\footnote{\ $\mathcal R_m:= \lbrace  1,\dots ,m\rbrace^2$ is the discrete full equivalence relation.}

\end{defi}

\begin{rem}
If $G$ is a \emph{minimal} ample groupoid \emph{with compact unit space $G^{(0)}$}, then $G$ has dynamical comparison if and only if for all nonempty clopen subsets $A,B\subseteq G^{(0)}$ satisfying $\mu(A)<\mu(B)$ for every $\mu \in M(G)$, then $A\precsim B$. Hence, it follows from \cite[Proposition~3.6]{1710.00393} that our notion of dynamical comparison generalizes Kerr's dynamical comparison at least in the ample case.  
\end{rem}

We now come back to the type semigroup $S(G)$, which is a useful tool to study dynamical comparison. This is due to the fact that dynamical subequivalence $A\precsim B$ translates to the inequality $[1_A]\leq [1_B]$ in the type semigroup.

Moreover, the invariant Borel measures on the unit space can be canonically identified with certain functionals on the type semigroup:

 For a preordered monoid $(S,+,\le )$ we denote by $F(S)$ the set of all unnormalized states on $S$, that is the set of all the monoid homomorphisms
$S\to [0,\infty]$. Note that $F(S)$ is a cone, i.e., we can sum and multiply by positive real numbers. If $x\in S$, we define the set of states on $S$ which are normalized at $x$ as 
$F(S,x)= \{ f\in F(S) : f(x)=1 \}$. This set might be empty, but in any case it is a convex subset of $F(S)$. 

Given an ample groupoid $G$, the set $\mathbb K$ of compact open subsets of $G^{(0)}$ is a ring of subsets of $G^{(0)}$, that is, it is closed under finite unions and relative complements (meaning that 
$E\setminus F\in \mathbb K$ if $E,F\in  \mathbb K$). The type semigroup $S(G)$ can now also be defined (see \cite[Proposition 7.3]{ABPS19}) as the commutative monoid with generators 
$[U]$ for each $U\in \mathbb K$ subject to the relations:
\begin{enumerate}
 \item $[\emptyset] = 0$,
 \item $[A\cup B]=[A]+[B]$ if $A,B\in \mathbb K$ and $A\cap B=\emptyset$,
 \item $[s(V)]=[r(V)]$ for each compact open bisection $V$ of $G$. 
\end{enumerate}
With this description, it is obvious that $F(S(G))$ is the set of all the finitely additive invariant positive measures on $\mathbb K$.  

We can now extend \cite[Lemma 5.1]{RS12} to groupoids as follows.

\begin{lem}
 \label{lem:RS-forgroupoids}
 Let $G$ be an ample second countable groupoid. 
 Then each $f  \in F(S(G))$ can be extended to a Borel invariant measure $\mu_f\in  UM(G)$. Moreover, the restriction of the measure $\mu_f$ to the open set
 $V:= \bigcup K$, where $K$ ranges over all the compact open subsets of $G^{(0)}$ such that $f([1_K]) <\infty$, is unique and regular.  
 \end{lem}

\begin{proof} As explained above, $\mathbb{K}$ is a ring of subsets of $G^{(0)}$. Therefore the set $\mathcal A$ of all the subsets $A$ of $G^{(0)}$ such that either $A$ or $A^{{\rm c}}$ belongs to $\mathbb K$ is an algebra of subsets of $G^{(0)}$
(i.e. it is closed under finite unions and complements). Note that all the members of $\mathcal A$ are clopen sets. In particular, since $G^{(0)}$ is second countable and totally disconnected 
each $A\in \mathcal A$ can be written as $A=\bigsqcup_{i=1}^{\infty} A_i$, where $A_i$ are compact open subsets of $G^{(0)}$. 

Given $f\in F(S(G)))$, we can define a premeasure $\mu$ on $\mathcal A$ by the rule $\mu (A) = f([1_A])$ if $A\in \mathbb K$, and 
$\mu (A) = \sum_{i=1}^{\infty} f([1_{A_i}])$ if $A=\bigsqcup_{i=1}^{\infty} A_i$ where $A_i$ are compact open subsets of $G^{(0)}$, if $A^{{\rm c}}\in \mathbb K$.
It is easy to check that the definition of $\mu (A)$ does not depend on the particular decomposition of $A$ into a disjoint union of a sequence
of compact open subsets of $G^{(0)}$, and that $\mu$ is a premeasure on $\mathcal A$. 

 By \cite[Theorem 1.14]{Foll84}, given $f\in F(S(G))$ there exists a positive Borel measure $\mu_f$ such that  $\mu_f (A) = f([1_A])$ for each $A\in \mathcal A$.
 In particular, this holds for every compact open subset $A$ of $G^{(0)}$. 
 To show that $\mu_f$ is invariant, take an open bisection $U$. Then since $G^{(0)}$ is second countable and totally disconnected, we can write
 $U= \bigsqcup_{i=1}^{\infty} U_i$, where $U_i$ are compact open subsets of $U$ (and thus compact open bisections).
 Now we get 
 $$\mu _f(s(U))= \sum_{i=1}^{\infty} \mu_f (s(U_i)) = \sum_{i=1}^{\infty} f([1_{s(U_i)}]) =  \sum_{i=1}^{\infty} f([1_{r(U_i)}]) = \sum _{i=1}^{\infty} \mu_f (r(U_i))= \mu_f (r(U)).$$
 This shows invariance of $\mu_f$.
 
 Now let $V= \bigcup  K $, where $K$ ranges over all the compact open subsets of $G^{(0)}$ such that $f([1_K]) <\infty$. Then $V$ is $\sigma$-finite, 
 and thus by \cite[Theorem 1.14]{Foll84}, there is a unique Borel measure
 $\ol{\mu}$ on $V$ such that $\ol{\mu}(K)=f([1_K])$ for all compact open subset $K$ of $V$. Hence the restriction of $\mu_f$ to $V$ is $\ol{\mu}$, and it is unique.  
 
 Now observe that every open subset of $V$ is $\sigma$-compact and that $\ol{\mu}(K)<\infty$ for each compact set $K$ of $V$. Hence, it follows from \cite[Theorem 2.18]{Rud87} 
 that $\ol{\mu}$ is a regular measure.  
  \end{proof}

 The following Lemma gives some justification that our definition of dynamical comparison is a sensible one for non-minimal groupoids.
 \begin{lem}
 \label{lem:dyn-comparison-forrestrictions}
 	Let $G$ be an ample groupoid. Then $G$ has dynamical comparison if and only if $G_U$ has dynamical comparison for every open $G$-invariant subset $U\subseteq G^{(0)}$.
 \end{lem}
\begin{proof}
	We only need to show that dynamical comparison passes to restrictions of $G$ to open $G$-invariant subsets. 
	Let $U\subseteq G^{(0)}$ be such an open $G$-invariant subset, and let $A,B\subseteq U$ be compact open subsets of $U$ such 
	that $A\subseteq r(G_U B)$. Moreover, assume that $\nu(A)<\nu(B)$ for every $\nu\in UM(G_U)$ such that $0<\nu (B)<\infty$. Let $\mu \in UM(G)$ such that $0<\mu (B)<\infty$. Then $\mu$ restricts to a measure $\mu_U\in UM(G_U)$ such that $0<\mu_U(B)<\infty$ and so $\mu (A) =\mu_U(A)<\mu_U(B)= \mu (B)$. Since $G$ has dynamical comparison and using again that $U$ is $G$-invariant, we obtain the desired conclusion.
\end{proof}
   

   
As mentioned in the introduction, in \cite{boenicke_li_2018,rainone_sims} it turned out that almost unperforation of the type semigroup is a very desirable property. 
Recall, that a preordered monoid $S$ is called \textit{almost unperforated} if whenever $x,y\in S$ and $n\in\N$ satisfy $(n+1)x\leq ny$, then $x\leq y$. One of the main goals of this paper is to relate almost unperforation of the type semigroup with certain properties of the underlying groupoid.
We can now relate stable dynamical comparison with almost unperforation of $S(G)$. 
   
\begin{lem}
 \label{lem:a-unp-implies-dyn-comparison}
 Let $G$ be an ample groupoid. If $G$ satisfies stable dynamical comparison, then $S(G)$ is almost unperforated. If we require additionally that $G$ is second countable, then the converse is also true.
 \end{lem}

 \begin{proof}
Suppose that $G$ satisfies stable dynamical comparison, and let $x,y\in S(G)$ be such that $(k+1)x\le ky$. Then there is some $m$ such that $x,y$ are represented by compact open subsets of 
$(G^m)^{(0)}$. Therefore since we are assuming that $G^m$ has dynamical comparison, we may assume that $m=1$. With this assumption we have $x=[1_A]$ and $y=[1_B]$, where $A$ and $B$ are compact open subsets of $G^{(0)}$.
Now we clearly have $A\subseteq r(GB)$ and $\mu (A) <  \mu (B)$ for each $\mu \in UM(G)$ such that $0<\mu (B) <\infty$. It follows from dynamical comparison that $[1_A]\le [1_B]$, as desired.

Now suppose $G$ is second countable and let $A,B$ be compact open subsets of $G^{(0)}$ such that $A\subseteq r(GB)$. Then there are compact open bisections $V_1,V_2,\dots ,V_m$ in $G$ such that $A=\bigsqcup_{i=1}^m r(V_i)$ and
$s(V_i)\subset B$. Therefore we get $[1_A]\le m[1_B]$. Assume in addition that $\mu (A) < \mu (B)$ for every measure $\mu$ such that $\mu (B)=1$.
Then by Lemma \ref{lem:RS-forgroupoids} we get that $f([1_A]) < f([1_B])$ for all $f\in F(S(G), [1_B])$. Now it follows from \cite[Proposition 2.1]{OPR12} that
there is some $k\in \N$ such that $(k+1)[1_A]\le k[1_B]$. Since $S(G)$ is almost unperforated, we get that $[1_A]\le [1_B]$, as desired.
The proof for $G^m$ is similar. (Note that $S(G^m)=S(G)$.)
   \end{proof} 

The above result begs the following natural question:
\begin{question}\label{Question}
	Is stable dynamical comparison equivalent to dynamical comparison?
\end{question}

The remainder of this section is dedicated to provide an affirmative answer of this question in the minimal setting, which then leads to the proof of Theorem \ref{Thm:A}.
In fact, in Proposition \ref{Lem:Weak comparison} we show that in the minimal setting, both notions are also equivalent to the following:

 \begin{defi}[\cite{1712.05129}]
     \label{def:weakdyncomparison}
    {\rm Let $G$ be an ample minimal groupoid with compact unit space. We say that $G$ satisfies {\it weak dynamical comparison} if 
     there exists a constant $C\geq 1$ such that whenever 
    $A,B\subseteq G^{(0)}$ are non-empty compact open subsets satisfying $\sup_{\mu\in M(G)} \mu(A)<\frac{1}{C}\inf_{\mu\in M(G)}\mu(B)$,  then $A\precsim B$.}
         \end{defi}
\begin{rem}\label{dc=wdc}
    For any ample groupoid $G$ with compact unit space, the set $M(G)$ of $G$-invariant regular Borel probability measures  
    is compact in the weak*-topology.
    Consequently, if $A$ and $B$ are compact open sets such that $\mu(A)<\mu(B)$ for all $\mu\in M(G)$, we can use the continuity of the function $\mu\mapsto \mu(B)-\mu(A)$ to see 
    that 
    $\inf_{\mu\in M(G)}(\mu(B)-\mu(A))\geq \varepsilon$ for some suitably small $\varepsilon >0$.
    \end{rem}
         
Before we state and prove the desired equivalence between the different notions in Proposition \ref{Lem:Weak comparison}, let us make the following elementary observation which plays a crucial role in the proof:
 \begin{lem}\label{Lem:Dimension Reduction}
	Let $G$ be a minimal ample groupoid with compact unit space. Suppose $G$ has weak dynamical comparison, with constant $C\ge 1$. Then, whenever $m\in \N$ and $A_1,\ldots, A_m,B\subseteq G^{(0)}$ are compact open subsets such that $$ \sup_{\mu \in M(G)} \Big(\sum_{i=1}^m \mu(A_i)\Big) < \frac{1}{2^{m-1}C} \inf_{\mu \in M(G)} \mu(B)$$ for all $\mu\in M(G)$, then $\bigsqcup_{i=1}^m A_i \times \lbrace i\rbrace \precsim B \times \lbrace 1\rbrace$.
\end{lem}
\begin{proof}
	The proof proceeds by induction on the number of levels $m$. The case $m=1$ is immediate from the fact that $G$ has weak dynamical comparison.
	Now if $m>1$ and the hypothesis of the lemma are fullfiled, then we have 
	$$ \sup_{\mu\in M(G)} \mu(A_m)\le \sup_{\mu \in M(G)} \sum_{i=1}^m \mu (A_i) < \frac{1}{2^{m-1}C} \inf_{\mu \in M(G)} \mu(B) < \frac{1}{C} \inf_{\mu\in M(G)} \mu (B),$$ so by weak dynamical comparison there exists $A_m'$ such that $A_m\sim A_m'\subseteq B$. 
	
	For each $\mu \in M(G)$, we have, by our assumption, 
	\begin{align*}
	\frac{1}{2^{m-2}C} \mu (B\setminus A_m') & = 	\frac{1}{2^{m-2}C} \mu (B) - \frac{1}{2^{m-2}C} \mu(A_m) \\
	& \ge \frac{1}{2^{m-2}C} \mu (B) - \mu(A_m)\\
	&  \ge \frac{1}{2^{m-2}C} \mu (B) - \frac{1}{2^{m-1}C} \mu (B)\\
	&  = \frac{1}{2^{m-1}C} \mu (B) \ge \frac{1}{2^{m-1}C}\inf_{\mu'  \in M(G)}  \mu' (B).
	\end{align*} 
	Hence we obtain
	$$\sup_{\mu \in M(G)} \sum_{i=1}^{m-1} \mu (A_i) <  \frac{1}{2^{m-1}C} \inf_{\mu \in M(G)} \mu (B) \le \frac{1}{2^{m-2}C} \inf_{\mu \in M(G)} \mu(B\setminus A_m'),$$
	Thus, we can apply the induction hypothesis so conclude that $\sqcup_{i=1}^{m-1} A_i\times \lbrace i\rbrace \precsim B\setminus A_m'\times \lbrace 1\rbrace$. Since we also had $A_m\sim A_m'\subseteq B$, the result follows.
\end{proof}

\begin{prop}\label{Lem:Weak comparison}
    Let $G$ be a $\sigma$-compact ample groupoid which is minimal and has a compact unit space. Then the following conditions are equivalent: 
    \begin{enumerate}
     \item $G$ satisfies stable dynamical comparison.
     \item $G$ satisfies dynamical comparison.
     \item $G$ satisfies weak dynamical comparison.
    \end{enumerate}
\end{prop}
    \begin{proof}
It is clear that ($1)\implies(2) \implies $(3).

(3)$\implies $(1).  Suppose that $G$ satisfies weak dynamical comparison and let $C\geq 1$ be such that whenever 
    $A,B\subseteq G^{(0)}$ are non-empty compact open subsets satisfying $\sup_{\mu\in M(G)} \mu(A)<\frac{1}{C}\inf_{\mu\in M(G)}\mu(B)$,  then $A\precsim B$.
    Given $m\ge 1$, we show that $G^m$ has dynamical comparison.

    Observe that, up to normalization, we can identify $M(G)$ and $M(G^m)$. Hence we will work with $M(G)$, with the understanding that 
each $\mu \in M(G)$ gives rise to the invariant measure on $(G^m)^{(0)}$ defined by
	$$\mu \Big( \bigsqcup_{i=1} ^m A_i\times \{ i\}\Big)  = \sum_{i=1} ^m \mu (A_i)$$
	for Borel subsets $A_i$ of $G^{(0)}$. 
	
    Suppose $A$ and $B$ are non-empty compact open subsets of $(G^m)^{(0)}$ such that $\mu(A)<\mu(B)$ for every $\mu\in M(G)$. By Remark \ref{dc=wdc}, there is $1\ge \varepsilon >0$  such that 
    $$\inf_{\mu\in M(G)}(\mu(B)-\mu(A))\geq \varepsilon.$$
    Since $G^m$ is $\sigma$-compact, we can find a countable cover $(V_n)_n$ of $G^m$ by compact open bisections. Let us use the shorthand notation $\theta_n$ for the corresponding homeomorphism $\theta_{V_n}:s(V_n)\rightarrow r(V_n)$.
    Now let $A_1=A\cap \theta_1^{-1}(B)$ and $B_1=\theta_1(A_1)$. For each $n>1$, define inductively
    $$A_n=(A\setminus(\bigcup_{i=1}^{n-1} A_i))\cap(\theta_n^{-1}(B\setminus\bigcup_{i=1}^{n-1} B_i)),$$
    and $B_n=\theta_n(A_n)$. Then all the sets $A_n$ and $B_n$ are (possibly empty) compact open disjoint subsets of $A$ and $B$ respectively. Moreover, we have $\mu(A_n)=\mu(B_n)$ for every $n\in\N$ and every $G$-invariant measure $\mu$.
    Consider the remainder sets
    $A_0=A\setminus (\bigcup_{n\geq 1} A_n)$ and $B_0=B\setminus (\bigcup_{n\geq 1} B_n)$.
    We clearly have $\mu(B_0)\geq \varepsilon$ for all $\mu\in M(G)$. Note also, that by construction, whenever $s(g)\in A_0$ for some $g\in G^m$, then $r(g)$ can not be an element of $B_0$. 
    So $r(G^mA_0)\subseteq (G^m)^{(0)}\setminus B_0$, or equivalently, $B_0\subseteq (G^m)^{(0)}\setminus r(G^mA_0)$. We claim that $\mu(A_0)=0$ for all $\mu\in M(G)$. It is enough to consider ergodic measures\footnote{Recall that $\mu\in M(G)$ is called \emph{ergodic} if every Borel $G$-invariant subset $E\subseteq G^{(0)}$ either has $\mu(E)=0$ or $\mu(G^{(0)}\backslash E)=0$. Similar to the group action case (see e.g. \cite[Theorem~8.1.8]{GKPT}), $\mu\in M(G)$ is ergodic if and only if it is an extreme point in $M(G)$. As its proof is essentially a repetition of the arguments used in the proof of \cite[Theorem~8.1.8]{GKPT}, we omit the details here.}.

    If we suppose $\mu(A_0)>0$, then $\mu(r(G^mA_0))>0$ and by ergodicity it follows that $\mu(B_0)\leq \mu((G^m)^{(0)}\setminus r(G^mA_0))=0$, a contradiction.
    Now for a fixed $\mu$ we have
    $$\lim_{n\rightarrow \infty}\mu(A\setminus(\bigcup_{i=1}^{n} A_i))=0.$$
    Since the limit is decreasing and the above measure values viewed as functions on the set $M(G)$ are continuous, the above convergence is uniform on $M(G)$ by Dini's Theorem. 
    Now let $\delta<\frac{\varepsilon}{2^{m-1}mC}$. Using the uniform convergence we conclude that there exists an $n_0$ such that for all $\mu\in M(G)$ we have
    $$\mu(A\setminus(\bigcup_{i=1}^{n_0} A_i))\leq \delta <\frac{\varepsilon}{2^{m-1}mC}\leq \frac{1}{2^{m-1}mC}\mu(B\setminus(\bigcup_{i=1}^{n_0} B_i)).$$
    We can further arrange $B\setminus(\bigcup_{i=1}^{n_0} B_i) \sim \bigsqcup_{j=1}^m D_j\times \{ j \}$ for clopen subsets $D_j$ of $G^{(0)}$
    such that $D_1\supseteq D_2\supseteq \cdots \supseteq D_m$. Then we get for every $\mu \in M(G)$ that 
    $$ \mu(A\setminus(\bigcup_{i=1}^{n_0} A_i))\leq \delta <\frac{\varepsilon}{2^{m-1}mC}\leq \frac{1}{2^{m-1}mC}\mu(B\setminus(\bigcup_{i=1}^{n_0} B_i)) \le \frac{1}{2^{m-1}C} \mu (D_1)$$
    and hence we can pass to the supremum on the left and infimum on the right to get
    $$\sup_{\mu \in M(G)} \mu(A\setminus(\bigcup_{i=1}^{n_0} A_i))\leq \delta< \frac{\varepsilon}{2^{m-1}mC}\leq \frac{1}{2^{m-1}C}\inf_{\mu\in M(G)}\mu(D_1).$$
    Now write $A\setminus(\bigcup_{i=1}^{n_0} A_i) = \bigsqcup_{j=1}^m C_j\times \{ j \}$, with $C_j$ a clopen subset of $G^{(0)}$. From the last inequality we conclude that
     $$ \sup_{\mu \in M(G)}\Big( \sum_{j=1}^m  \mu(C_j)\Big) =\sup_{\mu \in M(G)} \mu(A\setminus(\bigcup_{i=1}^{n_0} A_i))< \frac{1}{2^{m-1}C}\inf_{\mu\in M(G)}\mu(D_1).$$
    Now apply Lemma \ref{Lem:Dimension Reduction} to conclude that $A\setminus(\bigcup_{i=1}^{n_0} A_i) =\bigsqcup_{j=1}^m C_j \times \{j\} \precsim D_1\times \{1\}$. Since we also have
    $D_1 \times \{1\} \precsim  \bigsqcup_{j=1}^m D_j\times \{ j \} \sim  B\setminus (\bigcup_{i=1}^{n_0}B_i)$
    and $\bigcup_{i=1}^{n_0} A_i\precsim \bigcup_{i=1}^{n_0} B_i,$ we obtain $A\precsim B$ as desired.
\end{proof}

We can finally prove Theorem  \ref{Thm:A} stated in the introduction.

\begin{proof}[Proof of Theorem \ref{Thm:A}]
	Let us first assume that $G^{(0)}$ is compact. By Proposition \ref{Lem:Weak comparison}, dynamical comparison implies stable dynamical comparison for $G$. Hence the result follow from Lemma~\ref{lem:a-unp-implies-dyn-comparison}.
	
	If $G^{(0)}$ is just locally compact, we pick a compact open subset $K\subseteq G^{(0)}$. Since $G$ is minimal, $G$ and the restriction $G|_K$ are Morita equivalent. In particular, $G|_K$ is minimal itself and still has dynamical comparison. Indeed, suppose $A,B\subseteq K$ such that $\mu(A)<\mu(B)$ for all $\mu\in M(G|_K)$. Then if $\nu\in UM(G,B)$, by minimality of $G$ and compactness of $K$, $0<\nu(K)<\infty$. So the measure $\frac{1}{\nu(K)}\nu_{\mid K}\in M(G|_K)$. Hence $\nu(A)<\nu(B)$. Since $\nu$ was arbitrary we can use dynamical comparison for $G$ to conclude $A\precsim B$ in $G$. But then the compact open bisections implementing this subequivalence have range and source in $K$ since $A,B\subseteq K$. So we actually get $A\precsim B$ in $G|_K$ as desired. The result now follows from the fact that $S(G|_K)\cong S(G)$ \cite[Corollary 5.8]{rainone_sims}, and the first paragraph of this proof.
\end{proof}

\subsection{Absence of invariant measures}
We call a measure $\mu\in UM(G)$ trivial provided that $\mu(A)\in\lbrace 0,\infty\rbrace$ for all compact open subsets $A\subseteq G^{(0)}$.
In this short section  we characterize almost unperforation of the type semigroup when every measure in $UM(G)$ is trivial.

Recall that an element $x$ of a semigroup $S$ is called {\it properly infinite} if $2x\leq x$. 

\begin{prop}\label{Prop:purely infinite order units}
	Let $G$ be an ample second countable groupoid such that every measure in $UM(G)$ is trivial. Then $G$ has dynamical comparison if and only if every element in $S(G)$ is properly infinite. 
\end{prop}
\begin{proof}
	Suppose first that $G$ has dynamical comparison.
	We first consider the case of a $G$-full compact open subset $A\subseteq G^{(0)}$.
	Since every measure in $UM(G)$ is trivial it follows from Lemma \ref{lem:RS-forgroupoids} that $F(S(G),[1_A])=\emptyset$. By Tarski's Theorem (see for example \cite[Theorem~9.1]{Wagon}) we conclude that a multiple of $[1_A]$ is properly infinite. 
	By \cite[Theorem~4.3]{MR2806681} there exist order units $u,v\in S(G)$ such that $[1_A]=u+v$. Pick representative functions $u=[f]$ and $v=[g]$ and let $U=\supp(f)$ and $V=\supp(g)$. 
	Note, that the sets $U,V$ are both compact open and $G$-full subsets of $G^{(0)}$. Consequently, $[1_V]$ and $[1_U]$ are order units themselves by Lemma \ref{Lemma:Order units}, 
	so that $[1_A]\leq l [1_U]$ and $[1_A]\leq k[1_V]$ for some $k,l\ge 1$. This implies that $A\subseteq r(GU)\cap r(GV)$. Since there exist no $G$-invariant measures on $G^{(0)}$ which are non-trivial in $U$ or $V$, we may use dynamical comparison to conclude that in fact already $[1_A]\leq [1_U]$ and $[1_A]\leq [1_V]$.
	Putting everything together we compute
	$2[1_A]\leq [1_U]+[1_V]\leq u+v=[1_A]$ and reach our desired conclusion.
	
	Now if $[f]\in S(G)$ is an arbitrary order unit, write $f=\sum_{i=1}^n 1_{A_i}$ and let $A=\bigcup_{i=1}^n A_i$ be its support. Then $[1_A]$ is an order unit as well. Hence we can apply the first step above (multiple times) and conclude $2[f]\leq 2n[1_A]\leq [1_A]\leq [f]$.

	Suppose now that $[g]$ is an arbitrary element of $S(G)$. Let $I\subseteq S(G)$ be the order ideal generated by $[g]$. By Lemma \ref{Lemma:OrderIdeals} 
	we have that $I\cong S(G_U)$ for some $G$-invariant open subset $U\subseteq G^{(0)}$. 
	Note that $[g]$ is an order unit for $I=S(G_U)$, and that $G_U$ also has dynamical comparison (Lemma \ref{lem:dyn-comparison-forrestrictions}).
	Let $\mu  \in UM (G_U)$. Then we can extend $\mu $ to an invariant measure $\widetilde{\mu}\in UM(G)$ by the rule $\widetilde{\mu} (T)= \mu (T\cap U)$ for each
	Borel set $T$ of $G^{(0)}$. 
	It follows by our hypothesis that $\widetilde{\mu}$ is trivial and hence so is $\mu$.
	Therefore every measure in $UM(G_U)$ is trivial and so   
	it follows from the above argument that $[g]$ is properly infinite in $S(G_U)$ and hence also in $S(G)$.
	
	Conversely, assume that every element in $S(G)$ is properly infinite. Let $A,B\subseteq G^{(0)}$ be compact open subsets, such that $A\subseteq r(GB)$.
	Then $[1_B]$ is a properly infinite order unit in the order ideal $S(G_{r(GB)})$ of $S(G)$. It follows that there exists an $n\in \N$ such that $[1_A]\leq n[1_B]\leq [1_B]$, as desired.
		\end{proof}
\begin{rem}
Note that the equivalent properties in the previous proposition are also equivalent to every compact open subset of the unit space being $(2,1)$-paradoxical in the sense of \cite[Definition~4.5]{boenicke_li_2018}.
The reader might also want to compare these results with those obtained independently by Ma in section 5 of \cite{Ma21}.
\end{rem}

 The following generalizes \cite[Proposition~6.2]{Ma18} and gives an affirmative answer to Question \ref{Question} in the absence of interesting invariant measures.
\begin{prop}
	Let $G$ be an ample second countable groupoid such that every measure in $UM(G)$ is trivial. Then the type semigroup $S(G)$ is almost unperforated if and only if $G$ has dynamical comparison.
\end{prop}
\begin{proof} 
	One implication follows from  Lemma~\ref{lem:a-unp-implies-dyn-comparison}. Conversely, we assume that $G$ has dynamical comparison.
	Suppose we are given $[f],[g]\in S(G)$, such that $(n+1)[f]\leq n[g]$. It follows from Proposition \ref{Prop:purely infinite order units} that $[g]$ is properly infinite in $S(G)$.
	We conclude that
	$[f]\leq (n+1)[f]\leq n[g]\leq [g]$, as desired.
\end{proof}

\section{Almost Finite groupoids}\label{sec:AF}
In this section we study the type semigroups associated with almost finite groupoids.
Our main results reveal that almost finiteness is not strong enough of a condition to prove almost unperforation of the type semigroup in the non-minimal setting.
The reason for this lies in a different behaviour of the permanence properties of these two notions: almost unperforation passes to order ideals, while almost finiteness does not pass to restrictions of $G$ to arbitrary open invariant subspaces of $G^{(0)}$.
Prompted by this, we will show that a strong version of almost finiteness, which basically asks for every such restriction to be almost finite, indeed provides us with an almost unperforated type semigroup.

We use this characterization of almost unperforation to clarify the relationship between the type semigroup and the positive cone of the homology group $H_0(G)$.

\subsection{Definition and Properties}
We begin by recalling the definition of almost finiteness and proving some immediate consequences. 
\begin{defi}{\cite[Definition~6.2]{MR2876963}}\label{Def:AlmostFinite} Let $G$ be an ample groupoid with compact unit space.
\begin{enumerate}
    \item We say that $K\subseteq G$ is an \textit{elementary} subgroupoid if it is a compact open principal subgroupoid of $G$ such that $K^{(0)}=G^{(0)}$.
    \item Given a compact subset $C\subseteq G$ and $\varepsilon>0$, a compact subgroupoid $K\subseteq G$ with $K^{(0)}=G^{(0)}$ is called $(C,\varepsilon)$\textit{-invariant}, if for all $x\in G^{(0)}$ we have
	$$\frac{\abs{CKx\setminus Kx}}{\abs{Kx}}<\varepsilon.$$	
	\item We say that $G$ is \textit{almost finite} if for every compact set $C\subseteq G$ and every $\varepsilon>0$ there exists a $(C,\varepsilon)$-invariant elementary subgroupoid $K\subseteq G$.
\end{enumerate}
\end{defi}
\textbf{From now on, whenever we say that a groupoid $G$ is almost finite, we also assume that $G$ is ample and has a compact unit space.}
\begin{defi}\cite[Definition 3.2]{Suzuki}\label{def:castle}
	Let $K$ be a compact groupoid. A \textit{clopen castle} for $K$ is a partition $$K^{(0)}=\bigsqcup\limits_{i=1}^n \bigsqcup\limits_{j=1}^{N_i} F_j^{(i)}$$ into non-empty clopen subsets such that the following conditions hold:
	\begin{enumerate}
		\item For each $1\leq i\leq n$ and $1\leq j,k\leq N_i$ there exists a unique compact open bisection $V_{j,k}^{(i)}$ of $K$ such that $s(V_{j,k}^{(i)})=F_k^{(i)}$ and $r(V_{j,k}^{(i)})=F_j^{(i)}$.
		\item $$K=\bigsqcup\limits_{i=1}^n \bigsqcup\limits_{1\leq j,k\leq N_i} V_{j,k}^{(i)}.$$
	\end{enumerate}
	The pair $\left(\{F_j^{(i)}\mid 1\leq j\leq N_i\},\lbrace V_{j,k}^{(i)}\mid 1\leq j,k\leq N_i\rbrace\right)$ is called the $i$-th tower of the castle and the sets $F_j^{(i)}$ are called the levels of the $i$-th tower.
\end{defi}
\begin{rem}
	Note that the uniqueness of the bisections in $(2)$ above has some important consequences: If $\theta_{j,k}^{(i)}:F_k^{(i)}\rightarrow F_j^{(i)}$ denotes the partial homeomorphism corresponding to the bisection $V_{j,k}^{(i)}$, i.e. $\theta_{j,k}^{(i)}=r\circ (s_{\mid V_{j,k}^{(i)}})^{-1}$, then we have $(\theta_{j,k}^{(i)})^{-1}=\theta_{k,j}^{(i)}$,  $\theta_{j,k}^{(i)}\circ \theta_{k,l}^{(i)}=\theta_{j,l}^{(i)}$, and $\theta_{j,j}^{(i)}=id_{F_j^{(i)}}$.
\end{rem}

Recall that, as mentioned in \cite{Suzuki}, since compact ample principal groupoids always admit a clopen castle, Definition \ref{Def:AlmostFinite} is equivalent to the definition of almost finiteness given in \cite[Definition 3.6]{Suzuki}. We point 
out that due to this fact we will be using both equivalent notions of almost finiteness throughout the paper.

The following small lemma shows how to refine a castle as in Definition \ref{def:castle} and will be used frequently throughout the rest of this article:
\begin{lem}\label{Lemma:RefiningDecompositions}
	Let $K$ be a compact groupoid admitting a clopen castle. Given finitely many clopen subsets $A_1,\ldots, A_r\subseteq K^{(0)}$ there exists a clopen castle for $K$ such that every level of every tower of the castle is either contained in or disjoint from $A_l$ for every $1\leq l\leq r$.
\end{lem}
\begin{proof}
	Let us consider the case that we only have one clopen subset $A\subseteq K^{(0)}$. We will replace every tower of the castle by finitely many thinner towers, such that each level of the new towers is either contained in or disjoint from $A$. 
	Let $\theta_{j,k}^{(i)}:F_k^{(i)}\rightarrow F_j^{(i)}$ be the partial homeomorphism associated to the compact open bisection $V_{j,k}^{(i)}$. Consider the compact open subsets $\theta_{1,k}^{(i)}(A\cap F_k^{(i)})\subseteq F_1^{(i)}$ of the base of the $i$-th tower. Taking a clopen refinement we can find a decomposition $F_1^{(i)}=\bigsqcup\limits_{t=1}^{L_i} X^{(i)}_{t,1}$ such that each $X_{t,1}^{(i)}$ is either contained in or disjoint from every $\theta_{1,k}^{(i)}(A\cap F_k^{(i)})$.
	Let $X_{t,j}^{(i)}:=\theta_{j,1}^{(i)}(X_{t,1}^{(i)})\subseteq F_j^{(i)}$. Then we clearly have $K^{(0)}=\bigsqcup_{i=1}^{n}\bigsqcup_{t=1}^{L_i}\bigsqcup_{1\leq j\leq N_i} X_{t,j}^{(i)}$. Moreover, the sets $V_{j,k,t}^{(i)}:=V_{j,k}^{(i)}\cap s^{-1}(X_{t,k}^{(i)})$ are compact open bisections such that $s(V_{j,k,t}^{(i)})=X_{t,k}^{(i)}$ and $r(V_{j,k,t}^{(i)})=X_{t,j}^{(i)}$ and one easily checks, that
	$K=\bigsqcup_{i=1}^n \bigsqcup_{t=1}^{L_i} \bigsqcup_{1\leq j,k\leq N_i} V_{j,k,t}^{(i)}$. Hence we have constructed a finer clopen castle. In this new castle, for every $k$ we have $\theta_{1,k}^{(i)}(X_{t,k}^{(i)}\cap A)=X_{t,1}^{(i)}\cap \theta_{1,k}^{(i)}(A\cap F_k^{(i)})$. By construction, the latter set is either empty or all of $X_{t,1}^{(i)}$. Hence by applying $\theta_{k,1}^{(i)}$ we obtain that $X_{t,k}^{(i)}\cap A$ is either empty or all of $X_{t,k}^{(i)}$, as desired.
	Applying the above process successively to finitely many sets $A_1,\ldots, A_r$ yields the desired result.
\end{proof}

We continue this first part of the section showing important features and permanence properties of almost finiteness. To state them we need to recall some terminology, and well-known facts about almost finite groupoids:
\begin{enumerate}
	\item If $G$ is almost finite, then  $M(G)\neq \emptyset$ \cite[Lemma~3.9]{Suzuki}.
	\item If $G$ is almost finite and minimal, then $G$ is topologically principal \cite[Remark~3.10]{Suzuki}.
	\item If $G$ is an almost finite groupoid and $D\subseteq G^{(0)}$ is a closed $G$-invariant subset, then the restriction $G_{D}$ is almost finite \cite[Lemma 3.13]{Suzuki}.
	
	\item If $G$ admits a proper surjective groupoid homomorphism $\pi:G\rightarrow H$ onto an almost finite groupoid $H$, such that the restriction to every source fibre $Gx\rightarrow H\pi(x)$ is bijective, then $G$ is almost finite \cite[Lemma~5.1]{Suzuki}.
\end{enumerate}

We would like to add another crucial and natural permanence property of almost finiteness to the above list: invariance under stable isomorphism. Recall that two \'etale groupoids $G$ and $G'$ are \textit{stably isomorphic} if $G\times \mathcal R\cong G'\times \mathcal R$, where $\mathcal{R}=\N^2$ is the (discrete) full equivalence relation on $\N$. It is well-known that stable isomorphism agrees with Morita equivalence for ample (Hausdorff) groupoids with $\sigma$-compact unit spaces (see \cite[Theorem~2.19]{MR3601549}). In fact, there are a number of notions of equivalence for groupoids, and they all coincide for ample (Hausdorff) groupoids with $\sigma$-compact unit spaces (see \cite[Theorem~3.12]{FKPS18}).
\begin{lem}\label{Lemma:AF and Products}
	Let $G$ be an almost finite groupoid and $K$ be an elementary groupoid.
	Then $G\times K$ is almost finite.
\end{lem}
\begin{proof}
	Let $C\subseteq G\times K$ be a compact subset and $\varepsilon>0$. Then $C$ is contained in $\tilde{C}\times K$ for a compact subset $\tilde{C}\subseteq G$. 
	By almost finiteness of $G$, there exists a $(\tilde{C},\varepsilon)$-invariant elementary subgroupoid $\tilde{K}$ of $G$. Then $L:=\tilde{K}\times K$ is clearly an elementary subgroupoid of $G\times K$ and for every $(x,y)\in G^{(0)}\times K^{(0)}$ we have
	\begin{align*}
	\abs{CL(x,y)\setminus L(x,y)}& \leq \abs{(\tilde{C}\times K)(\tilde{K}\times K)(x,y)\setminus (\tilde{K}\times K)(x,y)}\\
	& \leq \abs{\tilde{C}\tilde{K}x\setminus \tilde{K}x}\abs{Ky}\\
	& <\varepsilon \abs{\tilde{K}x}\abs{Ky}=\varepsilon\abs{L(x,y)}.
	\end{align*}
\end{proof}

\begin{prop}\label{Prop:MoritaInvariance}
	Let $G$ and $G'$ be ample groupoids with compact unit spaces. Suppose that $G$ and $G'$ are stably isomorphic. Then $G$ is almost finite if and only if $G'$ is almost finite.
\end{prop}
\begin{proof}
	Let us fix an isomorphism $\Phi:G\times \mathcal{R}\rightarrow G'\times \mathcal{R}$. For each $n\in\N$ consider the clopen subgroupoid $G_n:=G\times \lbrace 1,\ldots, n\rbrace ^2$ of $G\times\mathcal{R}$.
	If $G$ is almost finite, then so is each $G_n$ by Lemma \ref{Lemma:AF and Products}.
	Let $H_n:=\Phi(G_n)\subseteq G'\times \mathcal{R}$. By definition, each $H_n$ is almost finite and $G'\times \mathcal{R}=\bigcup_n H_n$. Consider the compact open subset $W:=G'^{(0)}\times \lbrace 1\rbrace\subseteq G'\times \mathcal{R}$. Then $W$ is clearly a $G'\times\mathcal{R}$-full subset of $(G'\times\mathcal{R})^{(0)}=G'^{(0)}\times \N$ such that $(G'\times \mathcal{R})|_{W}\cong G'$. Hence it is enough to show that the restriction groupoid $(G'\times \mathcal{R})|_{W}$ is almost finite. But this follows from a slight adaptation of \cite[Lemma~3.12]{Suzuki}: If $C\subseteq (G'\times \mathcal{R})|_{W}$ is a compact subset and $\varepsilon>0$, then there exists an $n\in\N$ such that $C\cup W\subseteq H_n$. Using the compactness of $H_n^{(0)}$ and the fact that $W$ is $G'\times\mathcal{R}$-full, there exist finitely many compact open bisections $V_1,\ldots,V_l\subseteq G'\times \mathcal{R}$ such that $\bigcup_{i=1}^l s(V_i)=H_n^{(0)}$ and $r(V_i)\subseteq W$ for each $i$. For each $1\leq i\leq l$, we have $V_i\subseteq H_n$. Indeed, since
	$s(V_i),r(V_i)\subseteq H_n^{(0)}$ we have $s(\Phi^{-1}(V_i)),r(\Phi^{-1}(V_i))\subseteq G_n^{(0)}$. But then we must have $\Phi^{-1}(V_i)\subseteq G_n$, which implies our claim.
	
	Now let $\tilde{C}:=C\cup V_1\cup\ldots\cup V_l\subseteq H_n$ and use almost finiteness of $H_n$ to find a $(\tilde{C},\frac{\varepsilon}{2l})$-invariant elementary subgroupoid $K$ of $H_n$. Then we can literally copy the argument from \cite[Lemma~3.12]{Suzuki} to show that $K|_{W}$ is a $(C,\varepsilon)$-invariant elementary subgroupoid of $(G'\times \mathcal{R})|_{W}$. This completes the proof.
\end{proof}

\subsection{Almost finiteness and dynamical comparison} In this subsection we will study the implications of almost finiteness for the type semigroup of not necessarily minimal ample groupoids.
 The main observation is contained in the following Lemma, which says that the algebraic preorder on $S(G)$ is witnessed by the $G$-invariant measures on the unit space $G^{(0)}$.
The lemma is essentially a version of \cite{1710.00393} for ample groupoids obtained by combining techniques from \cite{MR2876963} and \cite{Suzuki}.
 
\begin{lem}
  \label{Lem:Comparison}
  Let $G$ be an almost finite groupoid and let $f,g\in C(G^{(0)},\Z)^+$. If $\mu (f) < \mu (g)$ for all $\mu\in M(G)$, then $[f] \le [g]$ in $S(G)$.  
   \end{lem}

 \begin{proof} Passing to $G^m$ for $m$ big enough, we can assume that $f=1_A$ and $g=1_B$ for clopen subsets $A$ and $B$ of $G^{(0)}$, with $\mu (A) <\mu (B)$ for all $\mu\in M(G)$.

Given a pair $(C,\varepsilon)$ we can find a $(C,\varepsilon)$-invariant elementary subgroupoid $K$ by almost finiteness. Using Lemma \ref{Lemma:RefiningDecompositions} we may assume that it admits a 
 castle $((F_j^{(i)})_{1\leq j\leq N_i},(V_{j,k}^{(i)})_{1\leq j,k\leq N_i})_{i=1}^n$, such that every level in every tower is either contained in or disjoint from each of the sets $A,B$.
	For each $1\leq i\leq n$ let $E_{i}=\lbrace k\mid F_k^{(i)}\subseteq A \rbrace$ and $F_{i}=\lbrace j\mid F_j^{(i)}\subseteq B\rbrace$ be the sets counting how many levels of the $i$-th tower 
	are contained in $A$ and $B$ respectively. Note that these sets depend on $(C,\varepsilon)$ (although we do not include this in our notation).
	
	\begin{claim*} There exists $(C,\varepsilon)$ such that for any $(C,\varepsilon)$-invariant elementary subgroupoid $K\subseteq G$  (admitting a castle for $K$ as described above), it follows that
\begin{equation}\label{Equation}
 \abs{E_{i}}\leq \abs{F_{i}}\,\,\,\text{ for each }1\leq i\leq n.
\end{equation}
\end{claim*}
\noindent \begin{claimproof}
	
	Suppose this is not the case. Then we can write $G$ as a directed union of symmetric compact subsets $C=C^{-1}$, and for each $\lambda:=(C,\varepsilon)$ find $(C,\varepsilon)$-invariant compact subgroupoids 
	$K_\lambda\subseteq G$ such that there exists a tower $\mathcal F_\lambda:=(F_{j}^{(i_\lambda)},V_{j,k}^{(i_\lambda)})_{1\leq j,k\leq L_\lambda}$ in the corresponding clopen castle for $K_\lambda$ with the property that
	 $$ \abs{E_{i_\lambda}}>  \abs{F_{i_\lambda}}.$$
	
	For each $\lambda$, let $x_{\lambda}$ be any element in  $F_1^{(i_\lambda)}$ (the basis of $\mathcal F_\lambda$), and define a probability measure $\mu_\lambda$ on $G^{(0)}$  by 
	$$\mu_\lambda(D)=\frac{1}{L_\lambda}\sum\limits_{j=1}^{L_\lambda} \delta_{x_\lambda}(\theta_{1,j}^{(i_\lambda)}(D\cap F_j^{(i_\lambda)})).$$
	Now let $U$ be a compact open bisection such that $U\subseteq C$, and note
        that $r(K_\lambda x_\lambda )=\lbrace \theta^{i_\lambda}_{j,1}(x_\lambda)\mid 1\leq j\leq L_\lambda\rbrace$ and $|K_\lambda x_\lambda |=L_\lambda$. Then, we get that : 
        $$\mu_\lambda(r(U))=\frac{\abs{r(U)\cap r(K_\lambda x_\lambda)}}{\abs{K_\lambda x_\lambda}}=\frac{\abs{U^{-1}K_\lambda x_\lambda}}{\abs{K_\lambda x_\lambda}}.$$
	Similarly, we get $$\mu_\lambda(s(U))=\frac{\abs{s(U)\cap r(K_\lambda x_\lambda)}}{\abs{K_\lambda x_\lambda}}=\frac{\abs{UK_\lambda x_\lambda}}{\abs{K_\lambda x_\lambda}}.$$
	Now $$\abs{UK_\lambda x_\lambda}=\abs{UK_\lambda x_\lambda\cap K_\lambda x_\lambda}+\abs{UK_\lambda x_\lambda\setminus K_\lambda x_\lambda}.$$
	Since $\abs{UK_\lambda x_\lambda\cap K_\lambda x_\lambda}=\abs{U^{-1}K_\lambda x_\lambda\cap K_\lambda x_\lambda}$, we also get $$\abs{U^{-1}K_\lambda x_\lambda}=\abs{UK_\lambda x_\lambda\cap 
	K_\lambda x_\lambda}+\abs{U^{-1}K_\lambda x_\lambda\setminus K_\lambda x_\lambda}.$$
	Putting all of this together we obtain
	\begin{align*}
	\abs{\mu_\lambda(s(U))-\mu_\lambda(r(U))}&=\frac{\abs{\abs{UK_\lambda x_\lambda\setminus K_\lambda x_\lambda}-\abs{U^{-1}K_\lambda x_\lambda\setminus K_\lambda x_\lambda}}}{\abs{K_\lambda x_\lambda}}\\
	&\leq2\frac{\abs{CK_\lambda x_\lambda\setminus K_\lambda x_\lambda}}{\abs{K_\lambda x_\lambda}}<2\varepsilon.
	\end{align*}
	Now let $\mu$ be a weak-$*$ cluster point of this net. Then $\mu\in M(G)$. Indeed, passing to a subnet, we can assume that $\mu = \lim_\lambda \mu_\lambda $. 
	Now if $U$ is any compact open bisection and $\varepsilon >0$ is arbitrary, 
	we can find  $\lambda = (C,\delta)$  such that $U\subseteq C$, $\delta <\frac{\varepsilon}{6}$ and moreover  $\abs{\mu (r(U))-\mu_\lambda (r(U))}<\varepsilon/3$
	and $\abs{\mu (s(U))-\mu_\lambda(s(U))}<\varepsilon /3$. By the above computation, we have $\abs{\mu_\lambda(s(U))-\mu_\lambda(r(U))}<2\delta < \varepsilon/3$. 
	
	Then 
	\begin{align*}
	\abs{\mu(s(U))-\mu(r(U))} & \leq \abs{\mu(s(U))-\mu_\lambda (s(U))}+\abs{\mu_\lambda(s(U))-\mu_\lambda (r(U))}\\+ & \abs{\mu_\lambda(r(U))-\mu(r(U))}
	 <\varepsilon.
	\end{align*}
	 As $\varepsilon>0$ was arbitrary, we conclude that $\mu(s(U))=\mu(r(U))$, and regularity implies that $\mu$ is indeed a $G$-invariant probability measure on $G^{(0)}$.
	
	Now, using the fact that each level $F_j^{({i_\lambda})}$ is either contained in or disjoint from each of the sets $A,B$ (Lemma \ref{Lemma:RefiningDecompositions}) we compute:
	\begin{align*}
	\mu(A)=\lim_\lambda \mu_\lambda(A)&=\lim_\lambda \frac{1}{L_\lambda}\abs{E_{i_\lambda}}\\
	&\geq \lim_\lambda \frac{1}{L_\lambda}\abs{F_{i_\lambda}}\\
	&=\lim_\lambda \mu_\lambda(B)=\mu(B).
	\end{align*}
	So we obtain that $\mu (A) \ge \mu(B)$, contradicting the hypothesis.   
		\end{claimproof}
		
	From the inequality (\ref{Equation}) we obtain injections of sets
	$$ E_{i} \hookrightarrow F_{i}, \qquad (1\le i\le n),$$
	and from this it is straightforward to see that $[1_A]\le [1_B]$ in $S(G)$. This concludes the proof.   
  \end{proof}

 As a first immediate application of this, we obtain an easy way to identify the order units in $S(G)$:
 \begin{cor}
 	Let $G$ be an almost finite groupoid. Then $[f]\in S(G)$ is an order unit if and only if $\mu(f)>0$ for all $\mu\in M(G)$.
 \end{cor}
 \begin{proof}
 	If $\mu(f)>0$ for all $\mu\in M(G)$, then by compactness of $M(G)$ there exists $N>0$ such that $1/N < \mu (f)$ for all $\mu \in M(G)$. Therefore, if $[g]\in S(G)$ is an arbitrary element, then there exists some $n\in \N$ such that $\mu(g)<n\mu(f)=\mu(nf)$ for all $\mu \in M(G)$.
 	By Lemma \ref{Lem:Comparison} we conclude that $[g]\leq [nf]=n[f]$.
 	Conversely, if $[f]\in S(G)$ is an order unit, then $[1_{G^{(0)}}]\leq N[f]$ for some $N\in \N$. Hence $1=\mu(1_{G^{(0)}})\leq N \mu(f)$, which implies our claim.
 \end{proof}
 We can now apply this result to come back to the study of almost unperforation of the type semigroup $S(G)$.
 We will denote by $S(G)^*$ the subsemigroup of $S(G)$ consisting of all the order-units of $S(G)$.
 In the following we will denote the algebraic preorder on $S(G)^*\cup\lbrace 0\rbrace$ by $\leq^*$. 
 We are now ready to prove our first main result in this section:
 
 \begin{thm}\label{Thm:Almost finite implies almost unperf}
 	If $G$ is almost finite, then $S(G)^*\cup\lbrace 0\rbrace$ is almost unperforated. In particular, if $G$ is almost finite and minimal, then $S(G)$ itself is almost unperforated.
 \end{thm}
 \begin{proof} 
 	Let $[f],[g]\in S(G)^*\cup\lbrace 0\rbrace$ such that $(n+1)[f]\leq n [g]$. We may assume $[g]\neq 0$ since the result is obvious otherwise. Then, for every $\mu\in M(G)$, we have $\mu(g)>0$ and $(n+1)\mu(f)\leq n\mu(g)$. We conclude that $\mu(f)<\mu(g)$ and hence $[f]\leq [g]$ in $S(G)$ by Lemma \ref{Lem:Comparison}. Thus, there exists some $[h]\in S(G)$ such that $[f]+[h]=[g]$. It follows that $\mu(h)=\mu(g)-\mu(f)>0$ for all $\mu\in M(G)$ and hence $[h]$ is an order unit by the previous Lemma. It follows that in fact we  have $[f]\leq^* [g]$ which completes the proof of the first statement.
 	
 For the second statement, we see that $S(G)$ is simple by Lemma~\ref{Lemma:Order units}. Hence, $S(G)=S(G)^*\cup\lbrace 0\rbrace$ is almost unperforated.
 \end{proof}
 
 This can be used to determine the groupoid homology of $G$ when combined with the following result:

 \begin{lem}\label{Lem:Cancellative}
 	 If $G$ is almost finite and no restriction $G_D$, for a closed invariant set $D\subseteq G^{(0)}$, is isomorphic to $\mathcal{R}_n$ for some $n\in \N$, then $S(G)^*\cup \lbrace 0\rbrace$ is cancellative.
 \end{lem}
 \begin{proof} We will use some results from \cite{AP96,AGPS,MR2806681}.
 	Recall that an element $x$ of a monoid $M$ is {\it weakly divisible} if it can be written as $x= 2a+3b$ for $a,b\in M$. If all order-units of $M$ are weakly divisible, then $M$ is said to have {\it weak divisibility for order-units} (\cite[Definition 2.2]{MR2806681}). An element $x$ of a conical monoid $M$ is said to be {\it irreducible} if it is nonzero and given any decomposition $x=a+b$ in $M$ we have that either $a$ or $b$ are zero.
 	If follows easily from \cite[Theorem 6.7]{AGPS} that an order-unit $u$ of a conical refinement monoid $M$ is weakly divisible if and only if $\ol{u}$ is not irreducible in any simple quotient $M/I$ of $M$.

 	Now, by Lemma \ref{Lemma:OrderIdeals}, any simple quotient of $S(G)$ is of the form $S(G)/S(G_{G^{(0)}\setminus D})\cong S(G_D)$ for a closed invariant subset $D\subseteq G^{(0)}$. Since $G_D$ is different from $\mathcal R_n$ and almost finite, it follows that $D$ is the Cantor set, implying the lack of irreducible elements in the simple quotients of $S(G)$. 
 	Therefore, all the order-units of $S(G)$ are weakly divisible and thus $S(G)$ has weak divisibility for order-units. 
 	
 	Now, we deduce from \cite[Theorem 3.4]{MR2806681} that $S(G)^*\cup \{0\}$ is a simple refinement monoid, and, by Theorem \ref{Thm:Almost finite implies almost unperf}, that $S(G)^*\cup \{ 0\}$ is almost unperforated. 
 	Hence, it follows 
 	from \cite[Theorem 3.8]{MR2806681} and \cite[Corollary 1.8]{AP96} that $S(G)^*$ is cancellative. To show that $S(G)^*\cup \{0\}$ is cancellative it is thus enough to show that  
 	 for a fixed element $u\in S(G)^*$ and $a\in S(G)^*\cup \{0\}$, the relation $u+a=u$ implies $a=0$. But this is obviously implied by the fact that $M(G)\ne \emptyset$, and the fact that 
 	 $\mu (x) >0$ for any order-unit $x$ in $S(G)$ and any $\mu\in M(G)$. 
 \end{proof}

 \begin{cor}\label{cancellative}
 	Let $G$ be a minimal almost finite groupoid. Then $S(G)$ is a cancellative monoid and $S(G)\cong H_0(G)^+$.
 \end{cor}
\begin{proof}
If $G$ is elementary i.e. $G\cong \mathcal{R}_n$ for some $n\in \N$, we have $S(G)=\N_0$ which is obviously cancellative. So let us assume that $G\ncong \mathcal{R}_n$. In this case, we apply Lemma~\ref{Lem:Cancellative} to obtain that $S(G)$ is cancellative. In both cases the result now follows from Proposition \ref{Prop:Cancellative hull of type semigroup}.
\end{proof}
 
 The results of this section so far indicate that almost finiteness itself does not lead to interesting properties of the whole type semigroup, but just to the subsemigroup of order units.
 
 This is largely due to the following fact:
 In contrast to the permanence property shown in Lemma \ref{lem:dyn-comparison-forrestrictions} for dynamical comparison, almost finiteness does not pass to the restrictions of $G$ to arbitrary compact open subsets of $G^{(0)}$ in general. In fact, we will build examples exhibiting this behaviour in section \ref{Section:CoarseGeometry}. To remedy this situation, we make the following definition:
 \begin{defi}\label{def:StronglyAF}
 	We say that an ample groupoid $G$ is {\it strongly almost finite} if the restriction $G|_A$ is almost finite for all compact open subsets $A$ of $G^{(0)}$.
 \end{defi}
We remark that our notion of strong almost finiteness should not be confused with  \cite[Definition 1.4]{1812.07511}, which is related but ultimately different.  
  
Clearly, every AF groupoid is strongly almost finite.
If $G$ is minimal and has a compact unit space, then our notion is equivalent to almost finiteness in the usual sense by Proposition \ref{Prop:MoritaInvariance}. However, in general, our notion is strictly stronger than almost finiteness. 

The remaining part of the section is dedicated to show that strong almost finiteness implies dynamical comparison of $G$ and almost unperforation of $S(G)$ (i.e. Theorem \ref{Thm:B}).   
 We need the following elementary lemma. Note that the lemma follows from \cite[Corollary 5.8]{rainone_sims} in case the set $U$ in its statement is $\sigma$-compact. 
 
 \begin{lem}
  \label{lem:Smorita-inv} Let $G$ be an ample groupoid, and let $B$ be a compact open subset of $G^{(0)}$. Then $S(G|_B)\cong S(G_U)$, where $U=r(GB)$ is the open invariant subset of 
  $G$ generated by $B$. 
   \end{lem}

 \begin{proof}
  We will use the generators and relations picture of the type semigroup introduced prior to Lemma \ref{lem:RS-forgroupoids} to define a semigroup homomorphism $\varphi \colon S(G_U) \to S(G|_B)$ as follows. Let $A$ be a compact open subset of $U=r(GB)$. Then there are compact open bisections $W_1,\dots , W_n$ such that
  $A= \bigsqcup_{i=1}^n r(W_i)$ and $s(W_i)\subseteq B$. Then set $\varphi ([A]) = \sum_{i=1}^n [s(W_i)]\in S(G_B)$. Suppose that $A= \bigsqcup_{j=1}^m r(W_j')$ 
  for compact open bisections $W_j'$ such that $s(W_j')\subseteq B$. Then for each $1\le i \le n$, we have  $s(W_i) = \bigsqcup_{j=1}^m \theta_{W_i}^{-1}(r(W_i)\cap r(W_j'))$     
 and for each $1\le j\le m$, we have $s(W_j') =\bigsqcup_{i=1}^n \theta_{W_j'}^{-1}(r(W_i)\cap r(W_j'))$. Moreover $[\theta _{W_i}^{-1}(r(W_i)\cap r(W_j'))] = 
 [\theta_{W_j'}^{-1}(r(W_i)\cap r(W_j'))]$ in $S(G_B)$. Therefore we get
 $$\sum_{i=1}^n [s(W_i)] = \sum _{i=1}^n \sum _{j=1}^m [\theta_{W_i}^{-1}(r(W_i)\cap r(W_j'))] = \sum_{j=1}^m \sum_{i=1}^n [\theta_{W_j'}^{-1}(r(W_i)\cap r(W_j'))]= \sum_{j=1}^m [s(W_j')].$$
 This shows that $\varphi ([A])$ does not depend of the particular decomposition of $A$. It is straightforward to show that $\varphi$ induces a semigroup homomorphism.
 Indeed, if $A\cap A'= \emptyset$, then we clearly get that $\varphi ([A\cup A']) = \varphi ([A]) + \varphi ([A'])$. If $V$ is a compact open bisection and $s(V)\subseteq U$, then 
 write $s(V)= \bigsqcup_{i=1}^n r(W_i) $ for compact open bisections such that $s(W_i)\subseteq B$. Then $r(V) = \bigsqcup_{i=1}^n r(VW_i)$ and $s(VW_i) = s(W_i)\subseteq B$.
 Therefore we obtain
 $$\varphi ([r(V)]) = \sum _{i=1}^n [r(VW_i)]=\sum _{i=1}^n [s(VW_i)] = \sum_{i=1}^n [s(W_i)] = \varphi ([s(V)]), $$
 and so the relation $[r(V)]= [s(V)]$ is also preserved by $\varphi$. 
 
 In the other direction, we can clearly define a homomorphism $\psi \colon S(G|_B)\to S(G_U)$ by $\psi ([D]) = [D]$ for a compact open subset $D$ of $B$. The maps $\varphi$ and $\psi$ are easily seen to be mutually inverse.
 This concludes the proof.
  \end{proof}

\begin{lem}\label{Lem:StrongComparison}
If $G$ is a strongly almost finite ample groupoid, then $G$ satisfies dynamical comparison. 
\end{lem}
\begin{proof}
	Let $A,B$ be compact open subsets of $G^{(0)}$ such that $A\subseteq r(GB)$, and assume that 
	$\mu (A) < \mu (B)$ for each $\mu \in UM(G)$ such that
	$0< \mu (B) <\infty$. We will show that $[1_A]\leq [1_B]$ in $S(G)$.
	
	Since $A\subseteq r(GB)$, there exist compact open bisections $V_1,V_2,\dots , V_m$ in $G$ such that $A=\bigsqcup_{i=1}^m r(V_i)$ and $s(V_i)\subseteq B$ for all $i$. 
	Now observe that $A\times \{ 1\}\sim D:= \bigsqcup_{i=1}^m s(V_i)\times \{ i \}$ within $G^m$. Note that $D$ is a compact open subset of $B^m= ((G|_B)^m)^{(0)}$,
	and that $(G|_B)^m$ is almost finite, because $G$ is strongly almost finite and almost finiteness is Morita invariant. 
	
	We next show that $\mu (D) < \mu (B\times \{ 1\})$ for all $\mu \in M((G|_B)^m)$. For this, we will use the results about Borel measures developed in Appendix \ref{sect:ext-measures}.
	
	Let $\mu \in M((G|_B)^m)$, and let $\mu '\in M(G|_B)$ be the invariant measure defined by
	$\mu ' (T) = m \mu (T\times \{ 1\})$. Then, by Proposition \ref{prop:extmeasUG} there exists $\widehat{\mu}\in UM(G)$ such that $\mu'(V)=\widehat{\mu}(V)$ for all open subsets $V\subseteq B$. Since $\widehat{\mu}$ is $G$-invariant, we have that 
	$$m\mu (D) = \sum_{i=1}^m \mu ' (s(V_i)) = \sum_{i=1}^m \widehat{\mu} (s(V_i)) = \sum _{i=1}^m\widehat{\mu}(r(V_i)) =\widehat{\mu} (A) < \widehat{\mu} (B) = \mu' (B) = m\mu (B\times \{ 1 \}),  $$
	as desired. Therefore we get that $\mu (D) < \mu (B\times \{ 1 \})$ for all $\mu \in M((G_B)^m)$. If we show that $D\precsim B\times \{1\}$ within $(G|_B)^m$, then clearly we will get that
	$A\precsim B$ within $G$.  
	
	Therefore, changing notation we can assume that $A,B$ are compact open subsets of $G^{(0)}$, that $B$ is $G$-full, and that $\mu (A) <\mu (B)$ for all $\mu \in M(G)$. In this situation, the result follows from
	Lemma \ref{Lem:Comparison}.
\end{proof}

We can now obtain our second main result of this section, i.e. Theorem \ref{Thm:B}.

 \begin{proof}[Proof of Theorem \ref{Thm:B}]
  By Lemma \ref{lem:a-unp-implies-dyn-comparison}, it suffices to show that $G$ satisfies stable dynamical comparison. Now by Lemma \ref{Lem:StrongComparison}, it suffices to show that $G^m$ is strongly almost finite
  for each $m\ge 1$.
  Let $B=\bigsqcup_{i=1}^m B_i\times \{ i \}$ be a compact open subset of $(G^m)^{(0)}$, where each $B_i$ is a compact open subset of $G^{(0)}$. Let $D=\bigcup_{i=1}^m B_i$.
  Then $D$ is a compact open subset of $G^{(0)}$ and clearly $(G^m)|_B$ and $G|_D$ are stably isomorphic. Hence $(G^m)|_B$ is almost finite by Proposition \ref{Prop:MoritaInvariance}.
  This shows that $G^m$ is strongly almost finite, and the proof is complete. 
   \end{proof}

\section{Coarse geometry}\label{Section:CoarseGeometry}
In this section we establish a new link between regularity properties in topological dynamics and coarse geometry. The starting point is the following recent result on the structure of amenable groups:
\begin{thm}\label{Theorem:Tilings-of-amenable-groups}\rm{(\cite{DHZ})}
	Let $\Gamma$ be a countable amenable group. Then $\Gamma$ admits an exact tiling into F\o lner sets of arbitrary invariance, i.e. for every finite subset $K\subseteq \Gamma$ and $\varepsilon>0$ there exist a number $n\in \N$, finite $(K,\varepsilon)$-invariant subsets $S_1,\ldots,S_n\subseteq \Gamma$ (the shapes) and  $F_1,\ldots, F_n$ of $\Gamma$ (the centers), such that $$\Gamma=\bigsqcup\limits_{i=1}^n\bigsqcup\limits_{\gamma\in S_i} \gamma F_i.$$
\end{thm}

Amenability for groups has a straightforward generalization to more general metric spaces. For the purposes of this work we restrict ourselves to those metric spaces $(X,d)$ with \textit{bounded geometry} (meaning that for any radius $R>0$ we have $\sup_{x\in X}\abs{B_R(x)}<\infty$) for reasons that will become clear shortly. To define amenability, we need the following notation:
For a finite subset $F\subseteq X$ we will write
$$\partial_R^+(F)=\lbrace x\in X\setminus F\mid d(x,F)\leq R\rbrace$$
for what is often called the \textit{outer $R$-boundary of $F$.}

\begin{defi}
	Let $X$ be a bounded geometry metric space. Then $X$ is called \emph{amenable} if for every $R>0$ and $\varepsilon>0$ there exists a finite set $F\subseteq X$, such that $\abs{\partial_R^+(F)}<\varepsilon \abs{F}$.
\end{defi}

A set $F$ as in the definition above is often referred to as an $(R,\varepsilon)$-F\o lner set. Now Theorem~\ref{Theorem:Tilings-of-amenable-groups} says that every amenable group does not just admit F\o lner sets of arbitrary invariance, but can be completely decomposed into F\o lner sets of arbitrary invariance.
The following definition is a version of the latter property for arbitrary metric spaces of bounded geometry.
\begin{defi}\label{tilings of arbitrary invariance}
	Let $X$ be a bounded geometry metric space. We say that $X$ \textit{admits tilings of arbitrary invariance}, if for all $R>0$ and $\varepsilon>0$ there exists a partition $X=\bigsqcup_{i\in I} X_i$ of $X$, such that $\abs{\partial_R^+(X_i)}<\varepsilon {\abs{X_i}}$ for all $i\in I$ and $\sup_{i\in I} diam(X_i)<\infty$.
\end{defi}

Let us illustrate this property by considering the following elementary example:
\begin{exam}
	We will show that the integers $\Z$ viewed as a discrete metric space with respect to the euclidean metric admits tilings of arbitrary invariance.
	The main point is that if $I$ is an interval in $\Z$ then the number $\abs{\partial_R^+(I)}$ is at most $2R$, and hence independent of the size and position of the chosen interval. Hence, given $R>0$, $\varepsilon > 0$, fix a natural number $N>\frac{2R}{\varepsilon}$ and partition $\Z$ into intervals $\Z=\bigsqcup_n I_n$ such that $\abs{I_n}=N$ for all $n\in \N$. Then each $I_n$ is an $(R,\varepsilon)$-F\o lner set by our choice of $N$ and $diam(I_n)\leq N$ since each $I_n$ is an interval, so we are done.
\end{exam}

Clearly, admitting tilings of arbitrary invariance is a very strong form of amenability.
As already explained, it was the tiling result for amenable groups that inspired the definition above. Indeed, every countable discrete group $\Gamma$ can be equipped with a proper left-invariant metric $d$ that is unique up to bijective coarse equivalence \cite[Lemma~2.1]{MR1871980}. The simplest examples are finitely generated discrete groups equipped with word metrics.

In particular, in the case of a countable discrete group equipped with any proper left-invariant metric, Theorem~\ref{Theorem:Tilings-of-amenable-groups} tells us that admitting tilings of arbitrary invariance is in fact equivalent to amenability of the group.

We are ready to establish the connection of this tiling property to regularity properties in topological dynamics. To this end, we use a construction of Skandalis, Tu, and Yu in \cite{zbMATH01761014}, which associates to every (discrete) metric space $X$ of bounded geometry a groupoid $G(X)$ 
over the Stone-\v{C}ech compactification $\beta X$ of $X$. Let us recall this construction:
For any radius $R\geq 0$ let $\Delta_R=\lbrace (x,y)\in X\times X\mid d(x,y)\leq R\rbrace$ be the $R$-neighbourhood of the diagonal in $X\times X$ and let $\overline{\Delta_R}$ denote its closure in $\beta(X\times X)$. Recall that we  identify any subset $S\subset X\times X$ as the corresponding set of principal ultrafilters in $\beta(X\times X)$. Then, as a set, one defines
$$G(X)=\bigcup_{R\geq 0} \overline{\Delta_R}.$$
Equip $G(X)$ with the weak topology it inherits from the union of compact open sets $\overline{\Delta_R}$ and with the groupoid structure it inherits as a subset of the pair groupoid $\beta X\times \beta X$.
It was shown in \cite[Proposition~3.2]{zbMATH01761014} that with the structure described above, $G(X)$ is a principal ample locally compact
$\sigma$-compact Hausdorff groupoid with $G(X)^{(0)}=\beta X$. We call $G(X)$ the \textit{coarse groupoid} associated to the metric space $X$.

The following is the main result of this section:
\begin{thm}\label{Thm:AlmostFiniteCoarseGroupoid}
	Let $X$ be a bounded geometry metric space. Then the following are equivalent:
	\begin{enumerate}
		\item $G(X)$ is almost finite,
		\item $X$ admits tilings of arbitrary invariance.
	\end{enumerate}
	In particular, $G(X)$ is strongly almost finite if and only if every subspace of $X$ admits tilings of arbitrary invariance.
	
\end{thm}

For the proof we need to recall some terminology and facts from \cite{MR3557774} and we are indebted to Rufus Willett for pointing us towards this article.
A \textit{partial translation} is a bijection $t:dom(t)\rightarrow ran(t)$ between two subsets $dom(t)$ and $ran(t)$ of $X$ such that $\sup_{x\in dom(t)} d(x,t(x))<\infty$. A partial translation $t$ is called \textit{compatible} with $\omega\in \beta X$ if $\omega(dom(t))=1$ (i.e. $\omega\in \overline{dom(t)}\subset \beta X$). Given $\omega\in \beta X$, and $t:dom(t)\to ran(t)\subseteq X$ a compatible partial translation, we use the notion of limit along the ultrafilter to define $$t(w):=\lim_\omega t\in \beta X.$$ In particular, for a fixed $\omega\in\beta X$, we say that an ultrafilter $\alpha\in \beta(X)$ is \textit{compatible} with $\omega$ if there exists a partial translation $t$ which is compatible with $\omega$ and satisfies $t(\omega)=\alpha$. We write $X(\omega)$ for the set of all $\alpha\in \beta X$ which are compatible with $\omega$. Note that there is a canonical bijection $F:X(\omega)\rightarrow G(X)_\omega$, given by $F(\alpha)=(\alpha,\omega)$.
The set $X(\omega)$ can be equipped with a canonical metric. Let $(t_\alpha)_{\alpha\in X(\omega)}$ be a compatible family of partial translations for $\omega$, i.e. each $t_\alpha$ is compatible with $\omega$ and $t_\alpha(\omega)=\alpha$. Then one can define
$$d_\omega(\alpha,\beta)=\lim\limits_{x\rightarrow\omega}d(t_\alpha(x),t_\beta(x)).$$
It was shown in \cite[Proposition~3.7]{MR3557774} that $d_\omega$ does indeed define a metric on $X(\omega)$ which does not depend on the choice of the compatible family.
Using this freedom in choosing the compatible family we observe the following:
\begin{lem}\label{Lemma:Metric} Let $\omega\in \beta X$ and $R\geq 0$. If $(\alpha,\omega)\in \overline{\Delta_R}$, then $d_\omega(\alpha,\omega)\leq R$.
\end{lem}
\begin{proof}
	Since $\overline{\Delta_R}$ is compact and open we may choose a compatible family such that $$\sup_{x\in X}d(t_\alpha(x),x)\leq R$$ for all $\alpha\in X(\omega)$ with $(\alpha,\omega)\in \overline{\Delta_R}$ and such that $t_\omega$ is the identity map on a suitable neighbourhood of $\omega$ in $\beta X$. It follows that
	$$d_\omega(\alpha,\omega)= \lim\limits_{x\rightarrow \omega} d(t_\alpha(x),t_\omega(x))=\lim\limits_{x\rightarrow \omega} d(t_\alpha(x),x)\leq R.$$
\end{proof}

In the following proof we will also use a different picture of the coarse groupoid via the pseudogroup of partial translations on $X$. To be more specific, for each partial translation $t:dom(t)\rightarrow ran(t)$ we can consider its extension $\bar{t}:\overline{dom(t)}\rightarrow \overline{ran(t)}$ to the respective closures in $\beta X$. Then $G(X)$ can also be realized as the quotient of $\{(\overline{t},\omega)\mid t \textit{ partial translation}, \omega\in dom(t)\}$ by the equivalence relation $(\bar{t}_1,\omega_1)\sim (\bar{t}_2,\omega_2)$ iff $\omega_1=\omega_2$, and $\bar{t}_1$ and $\bar{t}_2$ coincide on a small neighbourhood of $\omega_1=\omega_2$. We will denote the equivalence class of $(\bar{t},\omega)$ by $[\bar{t},\omega]$. The topology can be described by specifying a basis of compact open bisections as $U_{t}=\{[\bar{t},\omega]\mid \omega\in \overline{dom(t)}\}$. We refer the reader to \cite{zbMATH01761014} for further details.

\begin{proof}[Proof of Theorem~\ref{Thm:AlmostFiniteCoarseGroupoid}]
	Suppose first, that $G(X)$ is almost finite. Given $R>0$ and $\varepsilon>0$ we can find a $(\overline{\Delta_R},\varepsilon)$-invariant elementary subgroupoid $K\subseteq G(X)$. Let $\lbrace x_i\mid i\in I\rbrace$ be a family of representatives for the action of $K$ on $X$ and $X_i:=r(Kx_i)$. Then $X=\bigsqcup X_i$ and since $K$ is compact, we must have $K\subseteq \overline{\Delta_S}$ for some $S\geq 0$, from which it follows that $\sup diam(X_i)\leq S$. Since $K$ is $(\overline{\Delta_R},\varepsilon)$-invariant we get
	$$\abs{\partial_R^+(X_i)}=\abs{\overline{\Delta_R}K{x_i}\setminus K{x_i}}< \varepsilon\abs{K{x_i}}=\varepsilon \abs{X_i},$$ obtaining the desired implication.
	
	For the converse, let $C\subseteq G(X)$ be a compact subset and $\varepsilon>0$. By compactness of $C$ there exists an $R>0$ such that $C\subseteq \overline{\Delta_R}$. An application of condition (2), described in the statement, provides a partition $X=\bigsqcup_{i\in I} X_i$ of $X$, such that $\frac{\abs{\partial_R^+(X_i)}}{\abs{X_i}}<\varepsilon$ for all $i\in I$ and $\sup_{i\in I} diam(X_i)<\infty$. Define an equivalence relation $\mathcal{R}\subseteq X\times X$ by $x\mathcal{R}y$ if and only if there exists an $i\in I$ such that $x,y\in X_i$, i.e. $\mathcal{R}$ is precisely the equivalence relation which has the $X_i$ as its equivalence classes. Since the diameters of the $X_i$ are uniformly bounded, there exists an $S\geq 0$ such that $\mathcal{R}\subseteq \Delta_S$. We let $K$ be the closure of $\mathcal{R}$ in $\beta(X\times X)$. Then $K\subseteq\overline{\Delta_S}\subseteq G(X)$ is a compact open principal subgroupoid of $G(X)$ by construction. It remains to show, that $K$ is $(C,\varepsilon)$-invariant. For this we differentiate two situations:
	\begin{enumerate}
		\item  For $x\in X\subseteq G(X)^{(0)}$ fix $i\in I$ such that $Kx=X_i$. Then
		$$\abs{CKx\setminus Kx}\leq \abs{\Delta_RKx\setminus Kx}=\abs{\partial_R^+(X_i)}<\varepsilon \abs{X_i}=\varepsilon \abs{Kx};$$ hence, the claim follows.

	\item If $\omega\in \beta X\setminus X$ we need some more work.
	Let $(t_\alpha)_{\alpha\in X(\omega)}$ be a compatible family. Using that $K$ is compact and open in $G(X)$, we may (replacing finitely many $t_\alpha$, if necessary) assume that :
	
	\begin{itemize}
		\item for each $\alpha\in X(\omega)$ such that $(\alpha,\omega)\in K$ we have: $U_{t_\alpha}\subseteq K$, where $U_{t_\alpha}=\lbrace [\overline{t_\alpha},\gamma]\mid \gamma\in \overline{dom(t_\alpha)}\rbrace$ is a basic compact open bisection.
		\item $r(U_{t_\alpha})\cap r(U_{t_\beta}) =\emptyset$ whenever $\alpha\neq \beta$ and $(\alpha,\omega),(\beta,\omega)\in K.$
		\item $t_\omega$ is the identity on a neighbourhood of $\omega$.
		\item $s(U_{t_\alpha})=s(U_{t_\beta})$ for all $(\alpha,\omega),(\beta,\omega)\in K.$
	\end{itemize}

	Using that the map $\beta X\rightarrow \N$, given by $\omega\mapsto \abs{K\omega}$ is continuous (apply continuity of the Haar system on $G(X)$ to the characteristic function $1_K$), we may shrink the $U_{t_\alpha}$ further to assume that $\abs{Ky}=\abs{K\omega}$ for all $y\in dom(t_\alpha)$ for all $\alpha$ such that $(\alpha,\omega)\in K$.
	
	Now let $F: X(\omega)\rightarrow G(X)_\omega$ be the bijection from \cite[Lemma~C.3]{MR3557774}. Then apply \cite[Proposition~3.10]{MR3557774} to the finite set $F^{-1}(\overline{\Delta_R}K\omega)\subseteq X(\omega)$ to find a subset $Y\subseteq X$ with $\omega (Y)=1$, and for each $y\in Y$ an isometry $f_y:F^{-1}(\overline{\Delta_R}K\omega)\rightarrow X$ given by $f_y(\alpha)=t_\alpha(y)$. Then we claim that for all $y\in Y$ there exists a (unique) $i\in I$, such that $f_y(F^{-1}(K\omega))=X_i$.
	
	\begin{claimproof}
		
	Given $\alpha,\beta\in X(\omega)$ such that $(\alpha,\omega),(\beta,\omega)\in K$ we have that $[t_\alpha,y]\in U_{t_\alpha}\subseteq K$ and $[t_\beta, y]\in U_{t_\beta}\subseteq K$. Since $K$ is a subgroupoid, it follows that $[t_\beta \circ t_\alpha^{-1},t_\alpha(y)]\in K$. But this means that $(t_\beta(y),t_\alpha(y))\in\mathcal{R}$ and hence $f_y(\alpha)$ and $f_y(\beta)$ are in the same $X_i$.
	
	Conversely, we have $y,\omega\in \overline{dom(t_\alpha)}$ for every $\alpha\in F^{-1}(K\omega)$. Using our choice of the $t_\alpha$ we get $\abs{K\omega}=\abs{Ky}=\abs{X_i}$. Since $f_y$ is an injection defined on a finite set, our claim follows.	
\end{claimproof}

	Using Lemma \ref{Lemma:Metric} and the fact that $f_y$ is an isometry, it is easy to check that $f_y(F^{-1}(\overline{\Delta_R}K\omega\setminus K\omega))\subseteq \partial_R^+(X_i)$.
	Putting everything together we obtain
	\begin{align*}
	\abs{CK\omega\setminus K\omega}&\leq \abs{\overline{\Delta_R}K\omega\setminus K\omega}\\
	&=\abs{f_y(F^{-1}(\overline{\Delta_R}K\omega\setminus K\omega)}\\
	&\leq\abs{\partial_R^+(X_i)}\\
	&<\varepsilon \abs{X_i}=\varepsilon \abs{K\omega}.
	\end{align*}
\end{enumerate}
	This completes the proof of the first statement. For the second one, we first notice that $G(X)|_K=G(K\cap X)$ for every compact open subset $K$ of $\beta X$. In fact, this is the canonical one-to-one correspondence between subsets of $X$ and compact open subsets of $\beta X$. Hence, the second statement follows from the first.
\end{proof}
We have the following immediate consequence, which indicates that admitting tilings of arbitrary invariance is a useful notion from a coarse geometric point of view.
\begin{cor}\label{coarse strong af}
	Admitting tilings of arbitrary invariance is a coarse invariant.
	Moreover, if $f\colon X\rightarrow Y$ is a coarse equivalence between two bounded geometry metric spaces then $G(X)$ is strongly almost finite if and only if $G(Y)$ is strongly almost finite.
\end{cor}
\begin{proof}
	The first statement follows from Theorem~\ref{Thm:AlmostFiniteCoarseGroupoid} and Proposition \ref{Prop:MoritaInvariance}, once we note that coarsely equivalent metric spaces have Morita equivalent coarse groupoids (see \cite[Corollary~3.6]{zbMATH01761014}).
	The second statement follows from Theorem~\ref{Thm:AlmostFiniteCoarseGroupoid} and the first statement, since $A$ and $f(A)$ are coarsely equivalent for every $A\subseteq X$.
\end{proof}

Moreover, we can now reap the fruits of the additional work we put in to prove Theorem \ref{Thm:B} for not necessarily second countable groupoids to get the following immediate consequence:
\begin{cor}\label{coarse groupoid+almost unp}
Let $X$ be a bounded geometry metric space such that every subspace of $X$ admits tilings of arbitrary invariance. Then the type semigroup $S(G(X))$ of the associated coarse groupoid is almost unperforated.
\end{cor}

In the setting of countable discrete groups, we get the following result:
\begin{cor}\label{universal minimal}
	Let $\Gamma$ be a countable discrete group. Let $M\subseteq \beta \Gamma$ be the universal minimal $\Gamma$-space. Then the following are equivalent:
	\begin{enumerate}
		\item $\Gamma$ is amenable.
		\item $G(\abs{\Gamma})=\Gamma\ltimes \beta\Gamma$ is almost finite.
		\item $\Gamma\ltimes M$ is almost finite.
		\item $\Gamma\ltimes (\beta\Gamma\backslash \Gamma)$ is almost finite.
	\end{enumerate}
\end{cor}
\begin{proof}
	Suppose first that $\Gamma$ is amenable.
	 Applying the main result of \cite{DHZ}, we obtain an exact tiling of $\Gamma$ whose tiles are $(K,\varepsilon)$-invariant, i.e. we obtain a number $n\in \N$, finite subsets $S_1,\ldots,S_n\subseteq \Gamma$ and subsets $F_1,\ldots, F_n$ of $\Gamma$, such that $$\Gamma=\bigsqcup\limits_{i=1}^n\bigsqcup\limits_{\gamma\in S_i} \gamma F_i.$$
	This verifies the condition in Theorem \ref{Thm:AlmostFiniteCoarseGroupoid}, so $G(\abs{\Gamma})$ is indeed almost finite.
	If $G(\abs{\Gamma})$ is almost finite, then so is its restriction to the closed $\Gamma$-invariant subset $M\subseteq \beta\Gamma$ (see \cite[Lemma~3.13.]{Suzuki}). But $G(\abs{\Gamma})|_{M}\cong \Gamma\ltimes M$. Similarly, $(2)\Rightarrow (4)$.
	
	The implications $(3)\Rightarrow (1)$ and $(4)\Rightarrow (1)$ now follow from \cite[Proposition~4.7]{MR3694599}.
\end{proof}
\begin{exam}
Let $\Gamma$ be a countable discrete amenable group and $M\subseteq \beta \Gamma$ be the universal minimal $\Gamma$-space. Then $S(\Gamma\ltimes M)\cong H_0(\Gamma\ltimes M)^+$ is cancellative and almost unperforated by Corollary~\ref{universal minimal}, Corollary~\ref{cancellative} and Theorem~\ref{Thm:Almost finite implies almost unperf}.
\end{exam}
Let us now use the above characterization to treat another class of bounded geometry metric spaces that has attracted a lot of attention in geometric group theory, namely the so-called \textit{box spaces} associated to any countable discrete residually finite group.
Let us recall the relevant definitions: Suppose $\Gamma$ is a countable discrete residually finite group and $\sigma=(N_i)_{i\in\N}$ is a decreasing sequence of finite index normal subgroups of $\Gamma$ whose intersection $\bigcap_{i\in\N} N_i$ is trivial. Equip $\Gamma$ with a proper right-invariant metric $d$. For each $i\in\N$ let $\pi_i:\Gamma\rightarrow \Gamma/N_i$ be the canonical quotient map, 
and equip $\Gamma/N_i$ with the quotient metric. Then the box space $\Box_\sigma \Gamma$ is defined as the coarse disjoint union $\bigsqcup_{i}\Gamma/N_i$  (see e.g. \cite[Definition 6.3.2]{MR2562146}). In this setting, the following is our main result. 

\begin{prop}\label{prop: Box Spaces}
	Let $\Gamma$ be a countable discrete residually finite group with any nested decreasing sequence $\sigma=(N_i)_{i\in\N}$ of finite index normal subgroups of $\Gamma$. Then the following are equivalent:
	\begin{enumerate}
		\item $\Gamma$ is amenable;
		\item $\Box_\sigma\Gamma$ admits tilings of arbitrary invariance;
		\item $G(\Box_\sigma\Gamma)$ is almost finite.
		
	\end{enumerate}
	\end{prop}

\begin{proof}
{\rm ((1)$\Rightarrow$(2))}	Fix an arbitrary radius $R>0$, a tolerance $\varepsilon>0$ and a nested decreasing sequence $\sigma=(N_i)_{i\in\N}$.
	By a classical result of Weiss \cite{MR1819193} (see also \cite[Proposition~5.5]{MR2322178} for the version we are using), we can find a (large) number $i_0\in\N$ and a finite subset $T\subseteq\Gamma$ such that 
	\begin{enumerate}[(i)]
		\item $\Gamma=\bigsqcup_{\gamma\in N_{i_0}}T\gamma$ (i.e. $T$ is a monotile and $N_{i_0}$ is the set of tiling centers), and
		\item $\abs{\partial_R^+(T)}<\varepsilon \abs{T}$.
	\end{enumerate} 
	
	Moreover, by the definition of a box space and \cite[Lemmas~3.7,~3.11]{szabo_wu_zacharias_2017}, we may choose $i_1\geq i_0$ such that 
	\begin{enumerate}[(i)]\setcounter{enumi}{2}
		\item $d(\Gamma/N_{i},\Gamma/N_{j})>R$ for all $i\geq i_1$ and $j< i_1$, and
		\item for every $i\geq i_1$ the quotient map $\pi_i:\Gamma\rightarrow \Gamma/N_i$ has large isometry radii, in the sense that each $\pi_i$ is isometric on $B_{R+L}(\gamma)$ for all $\gamma\in N_{i_0}$, where $L:=\max\lbrace d(t,e)\mid t\in T\rbrace$ and $e$ the identity.
	\end{enumerate}
	Now for each $i\geq i_1$, let $C_i$ be a complete family of representatives for the quotient $N_{i_0}/N_i$.
	Set $X_0:=\bigsqcup_{i< i_1} \Gamma/N_i$. Then we have a decomposition
	\begin{equation}\label{Equation:Decomposition}
	\Box_\sigma\Gamma=X_0\sqcup \bigsqcup_{i\geq i_1}\bigsqcup_{c\in C_i}\pi_i(Tc).
	\end{equation}
	
	Note that the latter union is indeed disjoint by our choice of $C_i$ and property ${\rm (i)}$ above.
	
	We claim that for every $i\geq i_1$ and every $c\in C_i$, the quotient map $\pi_i$ restricts to an isometric bijection $\partial_R^+(Tc)\rightarrow \partial_R^+(\pi_i(Tc))$. Indeed, using right-invariance of the metric on $\Gamma$, we have $\partial_R^+(Tc)\subseteq B_{R+L}(c)$. It follows from item ${\rm (iv)}$ that $\pi_i$ is isometric on $B_{R+L}(c)$; hence, one has that $\pi_i(\partial_R^+(Tc))\subseteq \partial_R^+(\pi_i(Tc))$. To see the converse inclusion, let $xN_i\in \partial_R^+(\pi_i(Tc))$ and observe that $\partial_R^+(\pi_i(Tc))\subseteq \Gamma/N_i$ by item ${\rm (iii)}$. Also, let $t\in T$ such that $d(xN_i,tcN_i)=d(xN_i,\pi_i(Tc))$. Then we have
	$$R\geq d(xN_i,tcN_i)=\inf_{m,n\in N_i} d(xn,tcm)=\inf_{m,n\in N_i} d(xnm^{-1},tc).$$
	Therefore, there exists a $y\in\Gamma$ such that $xN_i=yN_i$ and $d(y,Tc)\leq R$. Clearly, we have $y\notin Tc$ and hence $y\in \partial_R^+(Tc)$ such that $\pi_i(y)=xN_i$. 
	
	Combining the above, that the metric on $\Gamma$ is right-invariant and item ${\rm (ii)}$, we obtain
	$$\abs{\partial_R^+(\pi_i(Tc))}= \abs{\partial_R^+(Tc)}=\abs{\partial_R^+(T)}<\varepsilon\abs{T}=\varepsilon \abs{\pi_i(Tc)}.$$
 Notice that by ${\rm (iii)}$ we also have $\partial_R^+(X_0)=\emptyset$, so every set in the decomposition (\ref{Equation:Decomposition}) is $(R,\varepsilon)$-invariant, as desired.
	
	Finally, combining item ${\rm  (iv)}$ and the fact that the metric on $\Gamma$ is right-invariant, we deduce $diam(\pi_i(Tc))=diam(Tc)=diam(T)$. So $$S:=\max\lbrace diam(T),diam(X_0)\rbrace$$ is a uniform bound on the diameters of the sets appearing in the decomposition (\ref{Equation:Decomposition}).

 {\rm ((2)$\Rightarrow$(3))} Provided by Theorem~\ref{Thm:AlmostFiniteCoarseGroupoid}. 
 
 	{\rm  ((3)$\Rightarrow$(1))} For ease of notation, we set $X:=\Box_{\sigma}\Gamma$. Then $G(X)\vert_{\beta X\backslash X}=({\beta X\backslash X})\rtimes \Gamma$ by \cite[Proposition~2.50 and Example~2.6]{MR3197659}. Since $\beta X\backslash X$ is a closed $G(X)$-invariant subset of $\beta X$, it follows that $({\beta X\backslash X})\rtimes \Gamma$ is almost finite as well. The claim now follows from \cite[Proposition~4.7]{MR3694599}.
\end{proof}
We conclude from the above Proposition that admitting tilings of arbitrary invariance is indeed a much stronger property than amenability if one considers metric spaces beyond groups: any box spaces are always (supr)amenable for rather trivial reasons. However, there exist many examples of finitely generated residually finite groups which are not amenable (e.g. the free groups $\mathbb{F}_n$ or $SL_n(\Z)$ for $n\geq 2$).

We will now proceed to present a construction that starting from any bounded geometry metric space $X$ produces another bounded geometry metric space $Y$ containing $X$, such that $Y$ admits tilings of arbitrary invariance. This will be very useful later in order to exhibit our examples.

\begin{prop}\label{Prop:BuildingAlmostFiniteGroupoids}
	Let $X$ be a discrete metric space with bounded geometry. Then the metric space $Y:=X\times \N$ with the graph metric (i.e. $d_Y((x,n),(y,m))=n+m+d_X(x,y)$ whenever $x\neq y$ and $d_Y((x,n),(x,m))=\abs{n-m}$) has bounded geometry and admits tilings of arbitrary invariance. 
	
	As a consequence, admitting tilings of arbitrary invariance does not imply Yu's property A.
\end{prop}
\begin{proof}
	We start by showing that $Y$ has bounded geometry. To this end let $B_R(x,k)\subseteq Y$ denote the ball of radius $R$ around $(x,k)\in Y$. Now given $R\geq 0$, we have $C:=\sup_{x\in X} \abs{B_R(x)}<\infty$ since $X$ has bounded geometry, and then $\abs{B_R(x,k)}\leq 2\lceil R\rceil \abs{B_R(x)}\leq 2\lceil R \rceil C$.
	Next we will prove that $Y$ admits tilings of arbitrary invariance. Let $R>0$ and $1\geq\varepsilon>0$ be given. Since increasing the radius $R$ only makes to problem harder, we may replace $R$ by $\lceil R \rceil $ and thus assume without loss of generality that $R\in\N$. Using that $Y$ has bounded geometry, there exists an $S\geq 0$ such that $\sup_{x\in X} \abs{B_R(x,0)}\leq S$. Now let $N>\frac{max\lbrace S, R\rbrace+R}{\varepsilon}$. For $x\in X$ and $k\in\N$ write $Y_{x,k}:= \lbrace x\rbrace\times \lbrace kN,\ldots ,((k+1)N)-1\rbrace$. Then we obtain a partition $$Y=\bigsqcup_{x\in X}\bigsqcup_{k=0}^\infty Y_{x,k}.$$
	The cardinality of each $Y_{x,k}$ is precisely $N$ and its diameter is $N-1$ independent of $x$ and $k$. It remains to show that the outer boundary of each of the sets in this partition is small relative to its cardinality. If $k\neq 0$ then $\abs{\partial_R^+(Y_{x,k})}\leq 2R<\varepsilon N=\varepsilon \abs{Y_{x,k}}$.
	If $k=0$, then we have $\partial_R^+(Y_{x,0})\subseteq B_R(x,0)\cup \left(\lbrace x\rbrace \times \lbrace N,\ldots,N+ R-1\rbrace\right)$. It follows that $\abs{\partial_R^+(Y_{x,0})}\leq S+ R  <\varepsilon N=\varepsilon \abs{Y_{x,0}}$.
	
The last statement follows from the fact that Yu's property A passes to subspaces (see \cite[Proposition~4.2]{MR1871980}).	
\end{proof}

Note that the above shows in particular that the property of admitting tilings of arbitrary invariance suffers the same shortcoming as amenability: It does not pass to arbitrary subspaces. Combined with Theorem~\ref{Thm:AlmostFiniteCoarseGroupoid} and using the identification $G(X\times\N)|_{ \bar{X}^{\beta(X\times \N)}}\cong G(X)$, our constructions show that almost finiteness for groupoids does not pass to restrictions to arbitrary compact open subsets.

Moreover, we can use it to produce a lot of examples which show that admitting tilings of arbitrary invariance is independent from other notions frequently studied in coarse geometry.
\begin{exam}\label{Examples for metric spaces}
	\begin{enumerate}
		\item Let $X$ be a bounded geometry metric space without Yu's property A. Then $Y=X\times \N$ defined as in Proposition~\ref{Prop:BuildingAlmostFiniteGroupoids} contains $X$ as a subspace by the construction. Hence, $Y$ admits tilings of arbitrary invariance and cannot have Yu's property A. Conversely, the free group on two generators $\mathbb{F}_2$ has Yu's property A, but can not admit tilings of arbitrary invariance, since it is non-amenable. Recall that $\mathbb{F}_2$ has asymptotic dimension one, so even finite asymptotic dimension does not imply tilings of arbitrary invariance.
		\item Let $X$ be a bounded geometry metric space which does not coarsely embed into a Hilbert space. Then $Y=X\times \N$ does not coarsely embed into a Hilbert space as well, but $Y$ admits tilings of arbitrary invariance.
	\end{enumerate}	
\end{exam}

The above examples are also very interesting when combined with Theorem~\ref{Thm:AlmostFiniteCoarseGroupoid}.
Most examples of almost finite groupoids known so far are amenable. In fact, for a transformation groupoid $\Gamma\ltimes X$ associated to a topologically free action of a discrete group $\Gamma$ acting on a totally disconnected compact space $X$, almost finiteness implies amenability of the acting group and a posteriori amenability of $\Gamma\ltimes X$ by \cite[Proposition~4.7]{MR3694599}. 
Our results yield new examples of almost finite groupoids which lack other desirable properties like amenability or a-T-menability. In particular, this shows that almost finiteness for general ample groupoids behaves very differently from the transformation groupoid case.

\begin{cor}\label{cor:almost finite non-amen}
	There exist almost finite ample principal groupoids $G$ which lack at least one of the following properties:
	\begin{enumerate}
		\item amenability,
		\item a-T-menability,
	\end{enumerate}
	Moreover, there exist ample groupoids with finite dynamic asymptotic dimension (see \cite{GWY17} for the definition) which are not almost finite.
\end{cor}
\begin{proof} To obtain the desired examples just take the coarse groupoid for the metric spaces described in Example~\ref{Examples for metric spaces} and combine them with the following facts:
\begin{enumerate}
    \item $G(Y)$ is amenable if and only if $Y$ has property A \cite[Theorem~5.3]{zbMATH01761014};
    \item $G(Y)$ is a-T-menable if and only if $Y$ coarsely embeds into a Hilbert space \cite[Theorem~5.4]{zbMATH01761014}; and
\end{enumerate}
For the final statement, we consider $\beta\F_2\rtimes \F_2=G(\F_2)$. From \cite[Theorem~6.4]{GWY17} we know that the dynamic asymptotic dimension of $G(\F_2)$ equals one, but as seen above $G(\F_2)$ is not almost finite.
\end{proof}

We should mention that Gabor Elek has independently found examples of non-amenable almost finite groupoids using a different approach (see \cite{1812.07511} for further details).

Finally, prompted by the results in section \ref{sec:AF}, we want to give some examples of strongly almost finite groupoids.

\begin{exam}\label{Example:Strongly almost finite}
	The coarse groupoids $G(\mathbb{Z})$ and $G(\mathbb{N})$ are strongly almost finite. In particular, their type semigroups are almost unperforated by Theorem \ref{Thm:B}.
	Let us focus on the case of the integers $\N$ (the result for $\Z$ follows the same line of argument by doing everything in two "directions").
	In view of Theorem \ref{Thm:AlmostFiniteCoarseGroupoid} it is enough to show that every subspace $A\subseteq \N$ admits tilings of arbitrary invariance.
	If $A\subseteq \N$ is bounded, it is finite and hence there is nothing to do. So let us assume that $A$ is unbounded.
	Write $A=\lbrace a_n\mid n\in \N\rbrace$ as an increasing sequence.
	Then there are two options: If $\sup_{n\in\N}\abs{a_n-a_{n+1}}<\infty$, then $A$ is coarsely equivalent to $\N$ itself and hence admits tilings of arbitrary invariance.
	We can deal with the remaining case $\sup_{n\in\N}\abs{a_n-a_{n+1}}=\infty$ by hand: Let $R>0$ and $\varepsilon>0$ be given. Let $N> \frac{2R}{\varepsilon}$.
	First, since the above supremum is infinite, we can find a subsequence $(a_{n_m})_m$ in $A$ such that $\abs{a_{n_m +1}-a_{n_m}}>R$ for all $m\in \N$ and $\abs{a_{n+1}-a_n}\leq R$ for all $n\not\in \lbrace n_m\mid m\in \N\rbrace$.
	Now let $A_1=\lbrace a_1,\ldots, a_{n_1}\rbrace$ and for $m>1$ we let $A_m:=\lbrace a_{n_{m-1}+1},\ldots, a_{n_m}\rbrace$.
	These sets form a disjoint partition of $A$ into $(R,\varepsilon)$-F\o lner sets such that $diam(A_m)\leq R\abs{A_m}$. So if $\sup_m \abs{A_m}<\infty$ we are done. This need not be the case however, so assuming that the sequence $(\abs{A_m})_m$ is unbounded, we need to refine our partition further.
	Now pick the subsequence consisting of all $A_{m_k}$ such that $\abs{A_{m_k}}\geq N\geq \frac{2R}{\varepsilon}$. We may assume that the maximal element of $A_{m_k}$ is strictly smaller than the smallest element of $A_{m_{k+1}}$ for all $k\in \N$.
	Now, writing $A_{m_k}$ as an increasing sequence we can easily partition each $A_{m_k}$ as $A_{m_k}=\bigsqcup_{l=1}^{L_{m_k}} B_{m_k,l}$, where $B_{m_k,1}$ consists of the first $N$ elements of $A_{m_k}$,  $B_{m_k,2}$ of the next $N$ elements and so on, such that $N= \abs{B_{m_k,l}}$ for all $1\leq l< L_{m_k}$ and $N\leq \abs{B_{m_k,L_{m_k}}}\leq 2N$.
	Then, we clearly have $$\frac{\abs{\partial_R(B_{m_k,l})}}{\abs{B_{m_k,l}}}\leq \frac{2R}{N}<\varepsilon.$$
	Thus, $$A=\bigsqcup_{m : \abs{A_m}\leq N} A_m\sqcup \bigsqcup_{k=1}^\infty \bigsqcup_{l=1}^{L_{m_k}}B_{m_k,l}$$ is a partition of $A$ into $(R,\varepsilon)$-F\o lner sets of diameter at most $2RN$.

\end{exam}

We finish this subsection by extending the above example to some groups with asymptotic dimension one. Recall the following definition:

\begin{defi}
	Let $X$ be a metric space. We say that the \emph{asymptotic dimension} of $X$ does not exceed $n$ and write asdim$(X)\leq n$ provided for every $R>0$ there exist $R$-disjoint families $\mathcal{U}^0,\ldots, \mathcal{U}^n$ of uniformly bounded subsets of $X$ such that $\cup_i \mathcal{U}^i$ is a cover of $X$.
\end{defi}

It is trivial to see that every bounded geometry metric space $X$ with asdim$(X)=0$ admits tilings of arbitrary invariance. Hence, its coarse groupoid $G(X)$ is strongly almost finite by Theorem~\ref{Thm:AlmostFiniteCoarseGroupoid}.

\begin{prop}
	Let $\Gamma$ be a finitely presented amenable group with asdim$(\Gamma)=1$. Then $G(\Gamma)$ is strongly almost finite.
\end{prop}
\begin{proof}
	By \cite[Theorem~2]{MR2431020} $\Gamma$ must be virtually cyclic. In particular, $\Gamma$ and $\Z$ are coarsely equivalent. From Corollary~\ref{coarse strong af} we only have to show that $G(\Z)$ is strongly almost finite, which is done in Example \ref{Example:Strongly almost finite}.
\end{proof}
\subsection{Non-amenable spaces}

Non-amenable metric spaces are well-studied in terms of their connections with properly infinite Roe-algebras \cite{ALLW1,ALLW2}. Using the type semigroup of the coarse groupoid and the dichotomy between amenability and paradoxicality for discrete metric spaces, we will in this section recover a celebrated Theorem by Block and Weinberger by a rather easy and conceptual proof.

\begin{prop}\label{zero homology non-amen}
	Let $X$ be a bounded geometry metric space and $G(X)$ be the coarse groupoid of $X$. Then the following are equivalent:
	\begin{enumerate}
		\item $X$ is non-amenable;
		\item every order unit in $S(G(X))$ is properly infinite;
		\item $H_0(G(X))=0$.
	\end{enumerate}
\end{prop}
\begin{proof}
	$(1)\Rightarrow (2)$ Since a non-amenable space admits a paradoxical decomposition (in the sense of \cite[Definition~2.5]{ALLW1}), the element $[1_{\beta X}]$ is properly infinite in $S(G(X))$ by \cite[Corollary~4.9]{boenicke_li_2018}. Now if $K\subseteq \beta X$ is $G(X)$-full, then $K\cap X$ is cobounded in $X$ and hence coarsely equivalent to $X$ itself. Since paradoxicality is a coarse invariant, $K\cap X$ is paradoxical itself and hence $[1_K]$ is properly infinite in $S(G(K\cap X))=S(G(X)|_K)\cong S(G(X))$.
	
	$(2)\Rightarrow (3)$: First of all $[1_{\bar{A}}]_0=0$ in $H_0(G(X))$ for all cobounded $A\subseteq X$, since a properly infinite element in $S(G(X))$ actually satisfies $2[1_A]=[1_A]$ and so we can just cancel in $H_0(G(X))$.
	Now if $A\subseteq X$ is arbitrary, follow the arguments in the proof of \cite[Lemma~5.4]{LW18} to get $[1_{\bar A}]_0=0$. The claim follows.
	
	$(3)\Rightarrow (1)$: Let $H_0(G(X))=0$ and suppose $X$ is amenable. Then there exists a $G(X)$-invariant Borel probability measure $\mu\in M(G(X))$. Denoting by $\hat\mu$ the corresponding functional on $H_0(G(X))$, it follows that $0=\hat\mu(0)=\hat{\mu}([1_{\beta X}])=\mu(\beta X)=1$, which is a contradiction.
\end{proof}

The remaining step is the identification of the $0$-th groupoid homology group of the coarse groupoid with the $0$-th uniformly finite homology group of $X$.
Recall that $H_n^{uf}(X)$ is obtained from the chain complex $(C_n^{uf}(X,\Z),\partial^{uf}_n)_n$, where $C_n^{uf}(X,\Z)$ consists of formal linear combinations $c=\sum_n c_{\bar{x}}\bar{x}$, where $\bar{x}$ denotes an $(n+1)$-tuple $(x_0,\ldots,x_n)\in X^{n+1}$, $c_{\bar{x}}\in \Z$ such that
\begin{enumerate}
	\item $c$ has finite propagation, in the sense that there exists a constant $P_c>0$ such that $c_{\bar{x}}=0$, provided that $\max d(x_i,x_j)\geq P_c$,
	\item and $c$ is bounded, meaning that  $\sup_{\bar{x}\in X^{n+1}}\abs{c_{\bar{x}}}<\infty$.
\end{enumerate}
The boundary map $\partial^{uf}_n:C_n^{uf}(X,\Z)\rightarrow C_{n-1}^{uf}(X,\Z)$ is defined on simplices by
$\partial^{uf}_n(x_0,\ldots,x_n)=\sum_{i=0}^n(-1)^i(x_0,\ldots,\hat{x_i},\ldots,x_n),$
where hat denotes omission of the term. One extends $\partial^{uf}_n$ to the whole of $C_n^{uf}(X,\Z)$ by linearity.
\begin{lem}\label{unniformly finite homology}
	There is a canonical isomorphism $H_0^{\textbf{uf}}(X,\Z)\cong H_0(G(X))$.
\end{lem}
\begin{proof}
	Using the universal property of the Stone-\v{C}ech compactification it is easy to see that there is a canonical linear bijection
	$$\Phi_0:C_0^{uf}(X,\Z)\rightarrow C(\beta X,\Z)$$
	given by extension of functions.
	Indeed, every element $c\in C_0^{uf}(X,\Z)$ can be viewed as a bounded (continuous) function $c:X\rightarrow \Z$ and can hence be extended to a bounded continuous function $\beta X\rightarrow \Z$. Restriction of functions clearly gives an inverse to $\Phi$.
	To complete the proof we need to check that the respective boundary maps are compatible.
	Similarly to the above observation, we can view a chain $c\in C_1^{uf}(X,\Z)$ as a bounded function $c:X\times X\rightarrow \Z$. Since $c$ has finite propagation, it is supported on $\Delta_{P_c}$. Again, we can extend $c$ continuously to a (compactly supported) function on $\overline{\Delta_{P_c}}\subseteq G(X)$, thus obtaining a well-defined linear map $\Phi_1:C_1^{uf}(X,\Z)\rightarrow C_c(G(X),\Z)$.
	Conversely, every function $f\in C_c(G(X),\Z)$ is bounded and its support is contained in $\overline{\Delta_R}$ for some $R>0$. Hence, restricting it to  to $\Delta_R$ (and extending by zero on $X\times X\setminus \Delta_R$) gives rise to a chain in $C_1^{uf}(X,\Z)$. One easily verifies that these constructions are inverse to each other.
	For $c\in C_1^{uf}(X,\Z)$ let $c_i=\sum_{\bar x} c_{\bar x} x_i$, $i=0,1$. Then we compute
	\begin{align*}
	    \Phi_0(\partial^{uf}_n(c)) & = \Phi_0(c_1)- \Phi_0(c_0)\\
	    & = s_*(\Phi_1(c))-r_*(\Phi_1(c))=\partial_1(\Phi_1(c)),
	\end{align*}
	where the second equation clearly holds when checking only for elements in $X\subseteq \beta X$ and hence by continuity on the whole of $\beta X$.
	Thus we have verified $\Phi(im(\partial_1^{uf}))\subseteq im(\partial_1)$, and a similar computation using the inverses of the $\Phi_i$ shows equality.
\end{proof}
The following corollary was first proved by Block and Weinberger in \cite[Theorem~3.1]{BW92}.
\begin{cor}\label{Cor:BlockWeinberger}
Let $X$ be a bounded geometry metric space. Then $X$ is non-amenable if and only if $H_0^{uf}(X,\Z)=0$.
\end{cor}
\begin{proof}
It follows directly from Proposition~\ref{zero homology non-amen} and Lemma~\ref{unniformly finite homology}.
\end{proof}

\appendix
\section{Extending measures}
\label{sect:ext-measures}

For a compact open subset $U$ of $G^{(0)}$, we would like to extend measures in $M(G|_U)$ to invariant (possibly unbounded) measures defined on the whole of $G^{(0)}$. 

Let $\mu \in M(G|_U)$.  We begin by defining a function $\rho$ on compact open subsets of $G^{(0)}$. First consider a compact open subset $K$ of the $G$-invariant open subset $Y:=r(GU)\subseteq G^{(0)}$ generated by $U$. Then we
can write $K= \bigsqcup_{i=1}^n r(W_i)$ for compact open bisections $W_1,\dots , W_n$ such that
$s(W_i)\subseteq U$. We then set
$$\rho (K)= \sum _{i=1}^n \mu (s(W_i)).$$
By the proof of Lemma \ref{lem:Smorita-inv}, we see that $\rho (K)$ does not depend on the particular decomposition as above.
It follows that $\rho (K_1\cup K_2) = \rho (K_1) + \rho (K_2) $ if $K_1,K_2$ are compact open subsets of $Y$ such that $K_1\cap K_2= \emptyset$.

If $K$ is a compact open subset of $G^{(0)}$ such that $K\nsubseteq Y$, then we set $\rho (K) = \infty$. The additivity formula above obviously holds also for 
any two compact open subsets $K_1$ and $K_2$ of $G^{(0)}$. 

The second step is to define $\rho$ for all open subsets of $Y$. Note that if $V$ is an open subset of $U$ then 
$$\mu (V) = \text{sup} \,\, \mu (K) ,$$
where the supremum is taken over all the compact open subsets of $V$. This follows from inner regularity of $\mu$ and the fact that $U$ is totally disconnected. 

Thus, for an open subset $V$ of $G^{(0)}$ it is natural to define
$$\rho (V) := \text{sup}\,\, \rho (K) ,$$
where $K$ ranges over all the compact open subsets of $V$. We then have that $\rho (V) = \mu (V)$ for each open subset $V$ of $U$.
It is trivial that $\rho (V)\le \rho (W)$ for $V$ and $W$ open subsets of $G^{(0)}$ such that $V\subseteq W$.
Note also that $\rho (V)=\infty $ for each open subset of $G^{(0)}$ such that $V\nsubseteq Y$. 

We now define an {\it outer measure} $\mu^*$ on $\mathcal P (G^{(0)})$ by 
$$\mu^* (A) = \text{inf} \Big\{  \sum_{j=1}^{\infty} \rho (V_j) : A\subseteq \bigcup_{j=1}^{\infty}  V_j,\,\,  V_j \text{ open subsets of } G^{(0)} \Big\}$$
(see \cite[Proposition 1.10]{Foll84}).

\begin{lem}
 \label{lem:equalitymurho}
 For any open subset $V$ of $Y$ we have $\mu^*(V) = \rho (V)$. In particular, we have $\mu^*(V) = \rho (V)=\mu (V)$ for any open subset $V$ of $U$. 
\end{lem}

\begin{proof}
 The inequality $\mu^*(V) \le \rho (V)$ is obvious. For the other inequality, take any compact open set $K$ such that $K\subseteq V$, and let $\{ V_j \}$ be a sequence of open subsets such that
 $V\subseteq \bigcup_{j=1}^{\infty} V_j$. Note that we have $K\subseteq \bigcup_{j=1}^N V_j$ for some $N\ge 1$. For each $x\in K$ we can select a small compact open neighborhood $W'_x$ of $x$ such that 
 $W'_x\subseteq V_l\cap K$ for some $1\le l\le N$. Since $K$ is compact 
 we can then write $K =\bigcup _{i=1}^M W'_{x_i}$ for some $x_i\in K$, $i=1,\dots , M$. Refining this decomposition, we can assume that the sets $W'_i$ are mutually disjoint.
 Collecting terms, we can write 
 $$K= \bigsqcup_{i=1}^N W_i,$$
 where $W_i$ are compact open subsets of $K$ such that  $W_i\subseteq V_i$ for $i=1,\dots , N$. We now have
 $$\rho (K) = \sum_{i=1}^N \rho (W_i) \le \sum_{i=1}^N \rho (V_i) \le \sum_{i=1}^{\infty} \rho (V_j).$$
 This shows that $\rho (V) \le \sum_{i=1}^{\infty} \rho (V_j)$. It follows that $\rho (V) \le \mu^*(V)$, as desired. 
\end{proof}

\begin{lem}
 \label{lem:KsubsetofV}
 Let $K$ be a compact open subset of an open subset $V$ of $Y$. Then $\rho(V) = \rho (K) + \rho (V\setminus K)$.  
\end{lem}

\begin{proof}
 Note that 
\begin{align*}
\rho (V)  & = \text{sup}\, \{ \rho (K'): K \subseteq K'\subseteq V, K' \text{ compact open }\}\\
& = \rho (K) +   \text{sup}\,  \{ \rho (K''): K''\subseteq V\setminus K,  K'' \text{ compact open }\}\\
& = \rho (K) +\rho (V\setminus K).  
 \end{align*}
\end{proof}
Recall that a subset $A$ of $G^{(0)}$ is called $\mu^*$-\textit{measurable} if
$$\mu^*(E) = \mu^*(E\cap A) + \mu^*(E\cap A^c)\qquad \text{ for all } E\subseteq G^{(0)} .$$
By Caratheodory's Theorem \cite[Theorem 1.11]{Foll84}, the collection $\mathcal M$ of all $\mu^*$-measurable sets is a $\sigma$-algebra, and the restriction of $\mu^*$ to $\mathcal M$ 
is a complete measure. 

It remains to show that all open subsets of $G^{(0)}$ belong to $\mathcal M$. Let $V$ be an open subset of $G^{(0)}$ and let $E\in \mathcal P (G^{(0)})$.
We only need to check that
$$\mu^*(E) \ge \mu^*(E\cap V) + \mu^* (E\cap V^c).$$
We can obviously assume that $\mu^*(E) <\infty $. In particular this implies that $E\subseteq Y$. 

Given $\varepsilon >0$, we can take a sequence $\{ V_j \}$ of open subsets of $G^{(0)}$ such that $E\subseteq \bigcup_{j=1}^{\infty} V_j$ and
$\sum_{j=1}^{\infty} \rho (V_j) - \mu^* (E) < \varepsilon $. In particular, we get that $V_j\subseteq Y$ for each $j$. 

Using Lemma \ref{lem:equalitymurho} and the fact that $\mu^*$ is an outer measure, we have 
$$\rho (\bigcup _{j=1}^{\infty} V_j) = \mu^* (\bigcup _{j=1}^{\infty} V_j) \le \sum_{j=1}^{\infty} \mu^*(V_j)= \sum _{j=1}^{\infty} \rho (V_j),$$ 
so we can indeed replace the sequence $\{V_j\}$
by just one term, namely $W:= \bigcup _{j=1}^{\infty} V_j$.

We can now take a compact open subset $K$ of $W$ such that
$$ \mu^* (W) = \rho (W) < \rho (K) +\varepsilon .$$
Now observe that $E\cap V\subseteq W\cap V$, and since $W\cap V$ is open, we get
$$\mu^*(E\cap V) \le \mu^* (V\cap W ) = \rho (V\cap W).$$
We can choose a compact open subset $K'$ of $Y$ such that $K'\subseteq V\cap W$ and 
$$ \rho (V\cap W) \le \rho (K') +\varepsilon .$$
Consider now $K'':= K\cap K'$, which is a compact open subset of $Y$ such that $K''\subseteq V\cap W$.
Note that 
\begin{align*}
\rho (K'') & = \rho (K\cap K') = \rho (K') - \rho (K'\setminus K\cap K')\\
& \ge \rho (K') -\rho (W\setminus K) = \rho (K') -(\rho (W)-\rho (K))\\
& > \rho (K') -\varepsilon,  
\end{align*}
where we have used Lemma \ref{lem:KsubsetofV} for the third equality and the inequality $\rho (W)-\rho (K)<\varepsilon$ for the last inequality. 

Now observe that $E\cap V^c \subseteq W\setminus K''$ because $K''\subseteq V$. Therefore
we have 
$$\mu^* (E\cap V^c) \le \mu^*(W\setminus K'') = \rho (W) - \rho (K'') < \rho (W) -\rho (K') +\varepsilon ,$$
and thus
$$\mu^*(E\cap V) +\mu^*(E\cap V^c) < \rho (K') +\varepsilon + \rho (W) - \rho (K') +\varepsilon = \rho (W) + 2\varepsilon \le \mu^*(E) + 3\varepsilon.$$
This shows the result.

Thus, we can obtain the following:
 
 \begin{prop}
  \label{prop:extmeasUG}
  Let $\mu \in M(G|_U)$, where $U$ is a compact open subset of $G^{(0)}$. Then there exists $\mu'\in UM(G)$ such that $\mu'(T)= \mu (T)$ for each Borel subset $T$ of $U$.  
   \end{prop}

 \begin{proof}
  By the above, we obtain a positive Borel measure $\mu'$ on $G^{(0)}$ such that $\mu'(V)= \mu (V)$ for each open subset $V$ of $G^{(0)}$.
  The measure $\mu'$ is just the restriction to the $\sigma$-algebra of Borel subsets of $G^{(0)}$ of the outer measure $\mu^*$ considered above. Since $\mu $ is a regular Borel measure we obtain that 
  $\mu'$ extends $\mu$. Finally if is clear from the definition of $\rho$ above that $\mu'$ is an invariant measure.  
   \end{proof}

	\ \newline
	{\bf Acknowledgments}. The second author would like to thank Rufus Willett for enlightening discussions and valuable suggestions on the content of section \ref{Section:CoarseGeometry} of the present article.
	Moreover, we are grateful to the anonymous referee for a very careful read and numerous suggestions for improvements.

\bibliographystyle{plain}

\end{document}